\documentclass{aims} 
\usepackage{amsmath}
\usepackage{paralist}
\usepackage[misc]{ifsym}
\usepackage{epsfig} 
\usepackage{epstopdf} 
\usepackage[colorlinks=true]{hyperref}
\hypersetup{urlcolor=blue, citecolor=red}
\allowdisplaybreaks

\def\currentvolume{X}
 \def\currentissue{X}
  \def\currentyear{200X}
   \def\currentmonth{XX}
    \def\ppages{X-XX}
     \def\DOI{10.3934/xx.xxxxxxx}

\theoremstyle{definition}

\title[Unstable manifold and periodic solutions]
{The strong  unstable manifold and periodic solutions in  differential delay  equations with cyclic monotone negative feedback} 

\author[Anatoli F. Ivanov and Bernhard Lani-Wayda]{}

\subjclass{Primary: 34K13; Secondary: 34K19, 92B25, 92C45.}
\keywords{monotone cyclic coupling,  Poincar\'e-Bendixson theorem, invariant manifold, periodic solutions, attractor location, gene regulatory systems.}

\thanks{$^*$Corresponding author: Bernhard Lani-Wayda}

\newcommand{\N}{\mathbb N}
\newcommand{\R} {\mathbb R}

\newcommand{\Co}{\mathbb C}

\newcommand{\dist}{\mbox{ dist}}

\newcommand{\graph}{\mbox{ graph}}

\newcommand{\resr}[1]{\raisebox{-1.2ex}{$\big| \raisebox{-0.3ex}{$#1$} $}}

\newcommand{\lmg}{\Big\{ } \newcommand{\rmg}{\Big\} }
\newcommand{\meig}{\;\big| \;}
\newcommand{\menge}[2]{\lmg {#1}\meig{#2}\rmg}

\newcommand{\Real}{\mbox{ \rm Re }}

\newcommand{\bc}{\begin{center}}\newcommand{\ce}{\end{center}}
\newcommand{\begitem}{\begin{itemize}} \newcommand{\itemend}{\end{itemize}}
\newcommand{\eqs}{\begin{equation*}\begin{split}}\newcommand{\eqsend}{ \end{split}\end{equation*}}

\newcommand{\pr}{\mbox{\rm pr}}

\newtheorem{thm}{Theorem}[section]
\newtheorem{cor}[thm]{Corollary}
\newtheorem{lem}[thm]{Lemma}
\newtheorem{prop}[thm]{Proposition}
\newtheorem{rem}[thm]{Remark}

\numberwithin{equation}{section}

\begin{document}
\maketitle

\centerline{\scshape Anatoli F. Ivanov$^{{\href{mailto:aivanov@psu.edu}{\textrm{\Letter}}}\,1}$ and 
Bernhard Lani-Wayda$^{{\href{mailto:Bernhard.Lani-Wayda@math.uni-giessen.de}{\textrm{\Letter}}}*2}$}

\medskip

{\footnotesize
 \centerline{$^1$Department of Mathematics,
Pennsylvania State University,
44 University Drive,
Dallas, PA 18612,
USA}
}

\medskip

{\footnotesize
 \centerline{$^2$Mathematisches Institut der Universit\"at Gie\ss en,
Arndtstr. 2,
35392 Gie\ss en,
Germany}
}
\bigskip

 \centerline{(Communicated by Gergely R\"{o}st)}


\begin{abstract}
For a class of $(N+1)$-dimensional systems of differential delay equations with a cyclic and monotone  negative feedback
structure, we construct a two-dimensional invariant manifold, on which phase curves  spiral outward towards a bounding periodic orbit.
For this to happen we assume essentially only instability of the zero equilibrium.  Methods of the  Poincar\'e-Bendixson theory due to
Mallet-Paret and Sell are combined with techniques  used by Walther for the scalar case $(N = 0)$.  Statements  on  the attractor
location and on parameter borders concerning stability and  oscillation are included. The   results apply to models for gene regulatory
systems, e.g. the `repressilator' system.
\end{abstract}


\section{Introduction}

The classical Poincar\'e-Bendixson theorem for  two-dimensional autonomous vector  fields
asserts  that the $\omega-$limit set of any initial value can  only consist of either a single   equilibrium, or a single  periodic orbit, or a heteroclinic chain which cyclically joins several equilibria. Versions of this theorem for higher dimensional vector fields were proved  by Mallet-Paret and Smith in \cite{MPSmith}, and  by Mallet-Paret and Sell for the  infinite dimensional dynamics generated by delay equations in \cite{MPS1} and  \cite{MPS2}. For these generalizations, monotonicity properties of the feedback functions are  essential, and have  the consequence that two-dimensional curves $\pi x$
obtained from the higher-dimensional phase curves $x$ (e.g., by  considering  $(x_i(t), x_{i+1}(t)) \in \R^2$ for some $i\in \{1,\ldots , N\}$ instead of $x(t) \in \R^n$)   have the property that  $\pi x$  and $\pi y $ do not intersect
 for  different globally defined phase curves $x$ and $y$, thus allowing for classical
 Poincar\'e-Bendixson-type arguments in the  plane.
A  discrete Lyapunov functional (roughly: counting sign changes of the solution) plays a main role in the arguments  from \cite{MPSmith}, \cite{MPS1} and  \cite{MPS2}, and the property that it is non-increasing along phase curves  is, again,  a consequence of monotonicity of the feedback.
Different oscillation classes of solutions can  be  distinguished according to the values of this Lyapunov functional.  The results from \cite{MPSmith},  \cite{MPS1} and  \cite{MPS2}  are frequently cited in  papers concerned with gene regulatory mechanisms, see e.g. \cite{HoriEtAl2010},  Proposition 1, p. 6, and p. 24, or \cite{HoriEtAl2013},   Section 3.1, p. 2583, or \cite{SamadEtAl}, p. 4406.

\medskip For scalar delay equations of the type $ \dot x(t) = - \mu x(t) + f(x(t-1)) $  with
$f$ monotonous\-ly decreasing and satisfying a negative feedback condition,  Walther
constructed a two-dimen\-sional  invariant manifold  with a boundary in  \cite{Wal91}, the boundary formed by a periodic orbit. This manifold is contained in the set $S$ of solution  segments  with at most one sign change, and can  be represented  as a graph over the two-dimensional leading eigenspace of
the linearization of the equation at the zero equilibrium.  This approach uses the set $S$
  (which corresponds to the value 1 of the  Lyapunov functional, and to the set  of the so-called slowly oscillating  solutions) and its invariance properties,
 but   not the discrete  Lyapunov functional explicitly.  In   later work \cite{Wal95}, Walther extended the result from   \cite{Wal91} to a description of the whole part of the attractor
within the class of  slowly oscillating  solutions, showing that it is  also a two-dimensional   graph over the same eigenspace (an extension of  the one just mentioned), with possibly several nested
periodic orbits.
A similar result for a positive monotone feedback situation was obtained by Krisztin, Walther and Wu in \cite{KriWaWu}, the part of the attractor described being three-dimensional here
(it was often called the `spindle').  In this work, the
zero-counting Lyapunov functional was also used.

\medskip
The approach from  \cite{Wal91} and  \cite{Wal95}  has the advantage that the position of
periodic orbits and an invariant manifold  in phase space (the space of continuous  functions  on
the delay interval $[-\tau, 0])$,  relative to  another significant object (the two-dimensional leading eigen\-space at zero)  is described. The  injective projection  from the
invariant manifold to this space takes place in phase space, and  is not of the nature of the maps $ \pi $  with values in $ \R^2$ mentioned above. On the other hand, these
results are restricted to  scalar equations, while the  methods from \cite{MPS1} and  \cite{MPS2}  have no restriction on the dimension of the system.

We briefly mention a selection of  previous results related to  periodic solutions
and the discrete Lyapunov functional:
The meanwhile classical paper  \cite{MP88} (Theorem D, p. 282)  on scalar delay equations with negative feedback obtained periodic solutions  in each oscillation class  of a Morse decomposition for  the attractor.
The argument there (p. 303 in \cite{MP88} was based on the Nussbaum fixed point index  \cite{Nussbaum}.
Cao \cite{Cao}  proved uniqueness of slowly oscillating periodic solutions for scalar negative feedback  equations as in \cite{MP88}, under the following condition in the nonlinearity $f$:
$$ x \mapsto x\cdot  f'(x) / f(x) \text{is strictly decreasing/increasing  for positive/negative } x. $$
In \cite{KriWa} it was shown that for  positive feedback  under additional hypotheses, in particular oddness of $ f$ and the  condition from \cite{Cao},  the `spindle'   from  \cite{KriWaWu} is actually the global attractor  and hence  the non-constant periodic orbit is unique.
In \cite{YiChenWu},  the special  case of  three-dimensional cyclic positive  feedback systems with the same nonlinearity in all three equations was considered, also under the condition from \cite{Cao}, and the $n-$dimensional extension sketched.   The existence of (unique, except for the value 2 of the Lyapunov functional, where  the proof of a symmetry property  of periodic solutions does not work) periodic solutions could be limited to finitely many  oscillation classes (Theorem 3.16, p. 47 there). The global attractor was then characterized as  the union of two stable equilibria, with the unstable  manifold of zero, and the unstable manifolds of all periodic orbits (Theorem 4.5, p. 49).
In \cite{GarabKrisztin},  scalar delay equations with  possibly positive or negative feedback and the condition from \cite{Cao} were considered,  and also  a special case of cyclic feedback system (with unidirectional coupling and  (as in \cite{YiChenWu} for the case of three equations)  the same feedback function in all equations).  In that  work,  Garab and Krisztin  obtained  uniqueness of slowly oscillating periodic solutions and asymptotic results on the dependence of their  period on the delay.
For a specific, piecewise linear scalar delay equation with possibly positive or negative feedback, more
 specific  information (in particular, on uniqueness)   on periodic solutions in various oscillation classes
(i.e., level sets of the Lyapunov functional)  was  obtained in \cite{Garab} (Theorem  4.3 on p. 2384).
In  both works \cite{Garab, Krisztin}and \cite{Garab},  phase plane arguments in the spirit of the
paper by Cao \cite{Cao} (reminiscent of Sturm-Liouville theory)  were important  for the uniqueness.

\medskip In the present paper we consider only systems with negative feedback.
Intuitively, in the situation of  the Poincar\'e-Bendixson theorem from
\cite{MPS1} and  \cite{MPS2}, if zero is the unique   equilibrium
 and if   solutions starting on the unstable manifold of zero cannot  come back in a homoclinic way, then their  $\omega-$limit set should be a periodic orbit. (In particular, there should exist a non-constant periodic orbit.)

We make this argument precise in the present paper,  combining  an analogue of the geometric approach in  phase  space from \cite{Wal91}  with the  `$(N+1)$-dimensional'  methods, in particular the
Lyapunov functional, from \cite{MPS1} and  \cite{MPS2}.
We obtain an analogue of the result from \cite{Wal91} for  systems with negative feedback
(an invariant two-dimensional  manifold,  an extension of   a  local `strong unstable' manifold,   with bounding periodic orbit).   The role of the set $S$  (functions with at most one sign change)  from \cite{Wal91} is taken by the level set $ \Sigma$  of the Lyapunov functional corresponding to level $1$.

Under the monotonicity assumptions, this  result gives, in particular, periodic  solutions  in situations  which had to be excluded in the papers  \cite{IvaBLW04} and \cite{IvaBLW20} (where, in turn, no monotonicity was assumed).
These situations are presence of several unstable eigenvalue pairs corresponding to slow oscillation,  or coexistence  of real  eigenvalues. (We comment on this
in more detail  after Theorem \ref{PerSolMfd}.)

\medskip
The general organization of the paper is as follows. Section 2 mainly introduces the reader
to results from  \cite{MPS1} and  \cite{MPS2},  with some  explanatory comments, and with some results on linear equations of continuous dependence type.
Section 3 constructs a local (`leading' or `strong') unstable manifold $W^{\mathrm{uu}}$ tangent at zero  to the corresponding eigenspace $X^{\mathrm{uu}}$,    and then shows that the closure $\overline{W}$ of its  continuation by the semiflow is  a  graph over $X^{\mathrm{uu}}$. All initial states in this manifold have a periodic  orbit as their $\omega-$limit set.
Section 4 gives  the main theorem which states, in particular, that the mentioned periodic
orbit is the same for all initial   states in  $\overline{W}$, and coincides with the `boundary'  $\overline{W}\setminus W$.  This theorem
does not  include earlier results from  \cite{AdH}, \cite{HT}, \cite{IvaBLW04}, \cite{IvaBLW20} on periodic solutions, which did not use monotonicity assumptions.
But these earlier results used a uni-directional feedback structure, and also assumed that the coupling of each  variable  $x_i $ to itself  is given by a simple decay term $ - \mu_i x_i$.
 In the present paper, the coupling structure can be more complex,  but monotonicity is required;
in this sense the results complement   the quoted earlier ones.

In Section 5 we  provide additional statements  on the rough location of the attractor (and thereby also the periodic orbit) for the  special case of systems mentioned in the last paragraph (with self-coupling  by a simple decay term). For these special systems, we also discuss the possible relationships  between the stability border and the oscillation border in terms of a parameter that represents the total feedback strength.

Cyclic gene regulatory systems with negative  feedback around the loop (in particular, the
so-called repressilator introduced by Elowitz and Leibler \cite{EloLeib})
  belong to the special class considered in Section 5.  Such systems  are treated as an illustrating application in the final Section 6.

\section{Cyclic monotone negative feedback systems, linearization and discrete Lyapunov functional}
\label{CyclicMono}

\subsection{Cyclic monotone feedback} In this section we consider, as in  \cite{MPS2}, an  $(N+1)$-dimen\-sion\-al
autonomous system of differential equations which  are, via an appropriate transformation,
 equivalent  to systems in the standard
feedback form
\begin{equation}
\left\{\begin{aligned}
\dot x_0  & = f^0(x_0, x_1), \; \\
\dot x_i  & = f^i(x_{i-1}, x_i, x_{i+1}), \; 1 \leq i \leq N-1, \\
\dot x_N  & = f^N(x_{N-1}, x_N, x_{0}(t- \tau)).
\end{aligned} \right. \label{standardfb}
\end{equation}
Here  the argument $t$ was omitted everywhere except in  the last equation,  where it appears with the delay $ \tau \geq 0$. This
corresponds to the  standard feedback form given by  formula (1.5), p. 443 in \cite{MPS2}, only that the delay is not normalized to 1.
We generally think of $ N \in \{1, 2,  3,   \ldots \}$, but, as   in  \cite{MPS1} and \cite{MPS2}
(see especially formula (1.2) on p. 386  of \cite{MPS1},  and the bottom of  p. 441 in \cite{MPS2}),  the case $N = 0$ is  included. In this  case, system  (\ref{standardfb}) should be read as the scalar equation
 $$\dot x_0(t) = f^0(x^0(t), x^0(t-\tau)).$$
In this special scalar case, our results coincide very much with those from   \cite{Wal91} and  \cite{Wal95}, but also allow a  nonlinear  dependence of $f^0$ on $x^0(t)$.

The results from \cite{MPS1} and \cite{MPS2}, that we will use in this paper, have to be read on
the history  interval $ [-\tau, 0],$  instead of $[-1,0],$  without  this normalization.
(Abbreviating  \eqref{standardfb} as $ \dot x (t)  = \vec{f}(x_0   \ldots , x_N, x_0(t-\tau))$,
the transformation $y(t) := x( \tau t)$  gives $\dot y (t) = \tau \vec{f}(y_0,   \ldots , y_N,y_0(t-1))$.)

We also assume that the functions $f^i$ in \eqref{standardfb} are of class $ C^1$, and the
monotonicity  conditions
\begin{equation}
\left\{\begin{aligned} \partial_2 f^0 &  >0, \\
\partial_1f^i &\geq 0, \; \text{ for } i \in \{1,  \ldots ,N\}, \\
\partial_3f^i & >  0, \; \text{ for } i \in \{1,  \ldots ,N-1\}, \\
\delta^*\cdot \partial_3f^N &  >0,  \end{aligned}  \right.
 \label{monotone}
\end{equation}
where all inequalities are assumed to hold for arbitrary values of their variables, and $\delta^* \in \{\pm1\}$  indicates either
the  positive or the negative nature of the feedback across the whole loop; we are interested in the case
where $ \delta^* = -1$ (negative feedback).
Following \cite{MPS2}, we call such a system a \textit{cyclic monotone feedback system in standard form.}
Transformations of the type $\tilde x_i(t) := - x_i(t)$  or $\tilde x_i(t) := x_i(t - \tau_i)$ may
equivalently bring systems, which do not a priori have it, to this standard form.
We also assume that $f^0(0,0) = 0 = f^i(0,0,0), \; 1 \leq i \leq N$, so that zero is an equilibrium solution for system \eqref{standardfb}. (This assumption will be strengthened
in  assumption \textbf{(A1)} below.)  Together with conditions \eqref{monotone}, one obtains that system \eqref{standardfb} satisfies the feedback inequalities from formulas (1.8) and (1.9) of \cite{MPS1}, p. 389. Note here the mistype in formula (1.9) of  \cite{MPS1} for the $g^i\;  (1 \leq i \leq N-1)$, where
$g^i(t,u,0,v) \geq 0$ in both cases $ u,v \geq 0 $ and $u,v  \leq 0$, and the correct form in
formula (2.2) of the same paper (with $ \delta_i = 1$ for  $1 \leq i \leq N-1$, and
$f^i(t,u,0,v) \geq 0$ in case $ u, v \geq 0$, but $f^i(t,u,0,v) \leq 0$ in case $ u, v \leq 0$).
Thus, system  \eqref{standardfb} is also a system in \textit{standard feedback form} in the sense of paper \cite{MPS1}, p. 389.

\subsection{Associated linear systems}
Consider now two solutions $x,y$ of \eqref{standardfb}, and their difference
$ \Delta: = y-x$.
Writing $f^i(y_{i-1}, y_i, y_{i+1}) - f^i(x_{i-1}, x_i, x_{i+1})$ as a sum of three differences, where
in each difference only one variable changes, and  setting
\[\begin{aligned} a_i(t) & := \int_0^1 \partial_1f^i[x_{i-1}(t) + s\cdot (y_{i-1}(t) - x_{i-1}(t)), y_i(t),  y_{i+1}(t)] \, ds, \\
   c_i(t) & := \int_0^1 \partial_2f^i[x_{i-1}(t), x_i(t) + s\cdot(y_i(t)- x_i(t)), y_{i+1}(t)] \, ds,\\
 b_i(t) & := \int_0^1 \partial_3f^i[x_{i-1}(t), x_i(t), x_{i+1}(t) + s\cdot (y_{i+1}(t)- x_{i+1}(t))] \, ds,
\end{aligned}
\]
we obtain
\begin{equation}\dot\Delta_i(t) = a_i(t) \Delta_{i-1}(t) + c_i(t) \Delta_i(t) + b_i(t) \Delta_{i+1}(t), \; 1 \leq i \leq N-1. \label{nonaut1} \end{equation}
With the analogous definition of $a_N(t), c_N(t), b_N(t)$, but replacing  $x_{N+1}(t) $ and $y_{N+1}(t)$ by
$ x_0(t-\tau)$ and $  y_0(t-\tau)$, we have
\begin{equation}\dot\Delta_N(t)  = a_N(t) \Delta_{N-1} + c_N(t) \Delta_N + b_N(t) \Delta_0(t - \tau). \label{nonaut2} \end{equation}
Finally,
\begin{equation}\dot \Delta_0(t)  = c_0(t) \Delta_0(t) + b_0(t) \Delta_1(t),
\label{nonaut3} \end{equation}
where $c_0(t) := \int_0^1 \partial_1f^0(x_0 + s\cdot (y_0- x_0), y_1) \, ds$ and
 $b_0(t) := \int_0^1 \partial_2f^0(x_0, x_1 + s\cdot (y_1- x_1)) \, ds$.
Here, again, the argument $t$ is omitted in all $\Delta_i,  x_i, y_i$, except where it appears as the delayed one $ t - \tau$.

From \eqref{monotone} in the negative feedback case, the  linear (in general, non-autonomous) system of differential equations for the $ \Delta_i$  has the properties
\begin{equation}
b_0(t)> 0, \; a_i(t) \geq 0, b_i(t)  >0 \text{ for }  i \in \{1,  \ldots ,N-1\}, \;a_N(t) \geq 0,  \, b_N(t) < 0,
\text{ for all } t. \label{coeffsigns}
\end{equation}
(The notation for the coefficient functions  here corresponds  to formula (3.4)  on p. 399 of
 \cite{MPS1}, except that  these functions do not have to be periodic here, and we use lower indices.)
 Analogous  systems  (also satisfying  \eqref{coeffsigns}) are similarly obtained from the variational equation
along a particular solution $x$ (then the integrals in the above definitions have to be
replaced by evaluation of the partial derivatives at that solution), and, as a special case, from the
linearization of system \eqref{standardfb} at an equilibrium solution. In that case, the constant coefficient system  reads as
\begin{equation}
\left\{\begin{aligned}\dot x_0(t)  & =  c_0 x_0(t)  + b_0x_1(t), \\
\dot x_i(t)   &= a_i x_{i-1}(t)  + c_ix_i(t)  + b_ix_{i+1}(t), \; 1 \leq i \leq N-1, \\
\dot x_N(t)   &= a_N x_{N-1}(t)  + c_Nx_N(t)  + b_N x_0(t - \tau)
\end{aligned}\right. \label{constcoeff}
\end{equation}
(compare formula (3.13), p. 401 in \cite{MPS1}).

\medskip It is generally  important that under the monotonicity conditions  \eqref{monotone},
 all  these three linear systems   obtained from any  one of the three procedures
\begin{equation} \left\{ \begin{aligned}
\text{1) }  & \text{ linearization at an equlibrium; }\\
\text{2) }  & \text{ linearization along a non-constant solution; } \\
\text{3) } &  \text{describing the difference between two solutions}
\end{aligned} \right. \label{3linsyst}
\end{equation}
satisfy conditions \eqref{coeffsigns}.
Thus these linear  systems have  the coupling structure and satisfy  the  feedback inequalities (in corrected form), which are
assumed in formulas  (1.8) and (1.9) of \cite{MPS1}, p. 389.

\subsection{State space and solutions of linear and nonlinear systems.}
Similar to  \cite{MPS1}, p. 394  and \cite{IvaBLW20}, p. 5379,   we use the state space $\mathbb{X} = C^0([-\tau, 0], \R) \times \R^N$
for system \eqref{standardfb}, and as well for systems of the type \eqref{nonaut1}-\eqref{nonaut3} or \eqref{constcoeff}. The
interpretation is that a solution $x$  starts with $x_0(t_0 + \cdot)\resr{[-\tau, 0]} = \varphi \in
 C^0([-\tau, 0], \R)$ and with $(x_1(t_0),   \ldots , x_N(t_0))   \in \R^N$.
 On $\mathbb{X}$   we  use the norm defined by $||[\varphi, x_1,   \ldots , x_N]||_{C^0} := \max\{||\varphi||_{\infty}, |x_1|,   \ldots , |x_N|\}$,
 which makes  $\mathbb{X}$ a
 Banach space (it corresponds to $C = C(\mathbb{K})$ in the notation  of \cite{MPS2}).
Occasionally we also use the stronger norm
 \[||\, [\varphi, x_1,   \ldots , x_N]\, ||_{C^1} := \max\{||\varphi||_{\infty}, ||\varphi'||_{\infty},  |x_1|,   \ldots , |x_N|\}\]
on the  subspace $ \mathbb{X} ^1 := \menge{[\varphi, x_1,   \ldots , x_N] }{ \varphi \in C^1([-\tau, 0], \R])}$ of $\mathbb{X}$.

For solutions $x:[t_0 - \tau, \infty) \to \R^{N+1} $  either of \eqref{standardfb}
or of one of the linear systems in \eqref{3linsyst}
with state $\psi  = [x_0(t_0 + \cdot)\resr{[-\tau,0]}, x_1(t_0),  \ldots , x_N(t_0)] \in \mathbb{X}$ at time $t_0$,
we use the notation
\[
x^{\psi}_t := [x_0(t + \cdot)\resr{[-\tau,0]}, x_1(t),  \ldots , x_N(t)]
= [(x_0)_t,  x_1(t),  \ldots , x_N(t)] \in \mathbb{X}
\]
for  the state of that solution at a time $ t \geq t_0$ (also for $t < t_0$, if the solution also  exists in backward time). Here, as
usual,  $(x_0)_t(\theta) = x_0(t + \theta), -\tau \leq \theta \leq 0$.  Complex-valued solutions of linear systems are defined
correspondingly, with  $\mathbb{X}_ {\Co}  := C^0([-\tau, 0], \Co) \times \Co^N$.

Note that for $\psi \in \mathbb{X}, t_0 \in \R$ and the corresponding solution $x$ of  any of the
systems under consideration with $x^{\psi}_{t_0} = \psi \in  \mathbb{X} $  one has
$x^{\psi}_{t_0+\tau} \in  \mathbb{X}^1$, if the solution exists at least on $[0,\tau]$
in forward time.

For the linear systems, one thus obtains linear evolution operators
\begin{equation}   U(t, t_0):  \mathbb{X} \to  \mathbb{X},
   \label{evop} \end{equation}
for all $t_0$ and $ t \geq t_0$ such that the coefficient functions of  the system are defined at least on the interval $[t_0, t]$.  Also, $  U(t, t_0)$ maps $\mathbb{X}$ into $\mathbb{X} ^1$ if $ t \geq \tau$.

The nonlinear  system  \eqref{standardfb} generates a semiflow $\Phi$ on the space $\mathbb{X}$,
given by $ (t,  \psi) \mapsto  x^{\psi}_t$  for $t \geq 0 $ and $ \psi \in  \mathbb{X}$  such that the corresponding solution  is defined on $[0, t]$.

We also note that a constant coefficient  equation of the form \eqref{constcoeff}
generates a $C^0$-semigroup $\{S(t)\}_{t \geq 0} $ of linear operators  $S(t) \in L_c(\mathbb{X}, \mathbb{X})$, which coincide with
the time-$t$-maps $\Phi(t, \cdot)$  for the case of this special  equation.

All three types of linear equations
mentioned above are of the form
\[
\dot x(t) = A(t) x(t) + B(t) x(t-\tau),
\]
with  $A(t), B(t) \in \R^{n \times n},$ where $ n = N+1$.
Here $ A(t) $ has a tridiagonal structure and $B(t)$ has only one nonzero entry,
compare \cite{MPS1}, p. 402. The subsequent  simple results on dependence of solutions of such equations
on $A$ and $B$ do not require cyclic feedback structure. We simply write $  |\;| $ for the $\max$-norm
on $ \R^n$, and also for the  induced norm on $\R^{n\times n}$.
  For continuous or $C^1$-functions $f$ on a compact  interval $I$, we shall use the $\max$-norm $||f||_{C^0} = ||f||_{C^0(I)} $, the $ C^1$
  norm $||f||_{C^1} = \max\{||f||_{C^0}, \,  ||f'||_{C^0}\}$,   and the $L^1-$norm $||f||_{L^1(I)} = \int_I |f| $.
\begin{lem}  \label{lem21} Assume $n \in  \N, \; t_0 \in \R$ and $A, \tilde A, B, \tilde B \in C^0([t_0, t_0 + \tau], \R^{n \times n})$,  and consider the solutions $x,y$ of
	\[\dot x(t) = A(t) x(t) + B(t) x(t-\tau), \quad \dot y(t) = \tilde A(t) y(t) + \tilde B(t) y(t-\tau)\]
	with the same initial value  $x(t_0 + \cdot)\resr{[- \tau, 0]} = y(t_0 + \cdot)\resr{[- \tau, 0]} = 	\varphi \in C^0([-\tau, 0], \R^n)$.
	
	There exist positive  constants $c_{A,B}, d_{A,B}, \tilde d_{A,B}$ which can be expressed in terms of  the numbers \\ $||A||_{C^0}, ||B||_{C^0}$  and $ \tau$  such that if \[\max\{ || \tilde A - A||_{C^0}, || \tilde B - B||_{C^0}\} \leq 1 \] then $y$ and  $\Delta := y-x $ satisfy
	\begin{align} ||y ||_{C^0([t_0,t_0 + \tau] )}&  \leq   c_{A,B}||\varphi||_{C^0},  \label{2.1.1} \\
	||\Delta ||_{C^0([t_0,t_0 + \tau] )}  & \leq   d_{A,B}(||A-\tilde A ||_{L^1} + ||B-\tilde B ||_{L^1})\cdot ||\varphi||_{C^0}, \label{2.1.2}\\
	||\dot \Delta ||_{C^0([t_0,t_0 + \tau] )} &\leq
 \tilde d_{A,B}(||A-\tilde A ||_{C^0} + ||B-\tilde B ||_{C^0})\cdot  ||\varphi||_{C^0}. \label{2.1.3}
\end{align}

	\end{lem}
\begin{proof}$x$ and $y$ satisfy the ODEs $\dot x(t) = A(t)x(t)  + f(t), \quad \dot y(t) = \tilde A y(t) + \tilde f(t)$, with
	$f(t) := B(t)\varphi(t-t_0-\tau)$ and  $\tilde f(t) := \tilde B(t)\varphi(t-t_0-\tau)$ for $t \in [t_0, t_0 + \tau]$.
	From the estimate in the last formula before formula (1.9) on p. 82  in Chapter III of  \cite{HaleODE} one sees that there exists
	a constant $c_{\tilde A}  >0$ expressible in terms of   $ ||\tilde A||_{C^0} $ such that
	\begin{align*} ||y||_{C^0([t_0, t_0 + \tau])}&  \leq  c_{\tilde A}  (\varphi(0) + 	||\tilde f ||_{L^1(t_0, t_0 + \tau)})
 \leq 	c_{\tilde A}  (\varphi(0) + 	\tau ||\tilde B||_{C^0}||\varphi||_{C^0} )\\
&  \leq 	   	c_{\tilde A}   (1 + 	\tau ||\tilde B||_{C^0}) ||\varphi||_{C^0}. \end{align*}
 Using
	$|| \tilde A - A||_{C^0} \leq 1 $,  the constant $c_{\tilde A}$ can be replaced  by a constant $c_A \geq c_{\tilde A}$.
	   Using also   $ || \tilde B - B||_{C^0}\} \leq 1$ we conclude
	\[ ||y||_{C^0([t_0, t_0 + \tau])} \leq \underbrace{ c_A[ 1 + \tau(1  +|| B||_{C^0}) ]}_{ =:c_{A,B}} \cdot ||\varphi||_{C^0},\] which proves \eqref{2.1.1}.
	  Inserting this inequality  in  estimate (1.10) from  \cite{HaleODE}, p. 83 and  noting that  $x(t_0) = y(t_0) = \varphi(0)$,
	  one obtains a constant $d_A >0$ depending on $||A||_{C^0}$ such that
	 \[ \begin{aligned} ||\Delta||_{C^0([t_0, t_0 + \tau])}
&\leq d_A  ||y||_{C^0([t_0, t_0 + \tau])}  \cdot ||A-\tilde A ||_{L^1} +  ||f-\tilde f ||_{L^1}\\
&\leq d_A \cdot  c_{A, B}|| \varphi||_{C^0} \cdot ||A-\tilde A ||_{L^1} +  ||f-\tilde f ||_{L^1}\\
	 &\leq   d_A \cdot  c_{A, B}|| \varphi||_{C^0} \cdot ||A-\tilde A ||_{L^1} + ||B- \tilde B||_{L^1} \cdot ||\varphi||_{C^0}\\
	 &\leq \underbrace{\max\{  d_A \cdot c_{A, B}, 1\}}_{ =:d_{A,B}} \cdot(||A-\tilde A ||_{L^1} + ||B- \tilde B||_{L^1})\cdot  ||\varphi||_{C^0}.
	 	 \end{aligned} \]
	 	 Estimate \eqref{2.1.2} follows  with the indicated choice of $d_{A,B}$.
Further,  using  \eqref{2.1.1}, \eqref{2.1.2} and that
$ ||\; ||_{L^1} \leq  \tau ||\; ||_{C^0}$, one obtains  for  $ t \in [t_0, t_0 + \tau]$
\[
\begin{aligned} |\dot \Delta(t)| & = |A(t) x(t) + B(t) \varphi(t-\tau) - \tilde A(t) y(t) - \tilde B(t)
 \varphi(t-\tau)| \\
& = |A(t)(x(t) - y(t)) + (A(t) - \tilde A(t)) y(t) + (B(t) - \tilde B(t))\varphi(t- \tau) | \\
&\leq ||A||_{C^0}\cdot  ||\Delta||_{C^0([t_0, t_0 + \tau])}  +
||A-\tilde A ||_{C^0} \cdot||y||_{C^0([t_0, t_0 + \tau])} + ||B-\tilde B||_{C^0}||\varphi||_{C^0} \\
&\leq ||A||_{C^0}  d_{A,B}(||A-\tilde A ||_{L^1} + ||B-\tilde B ||_{L^1})\cdot ||\varphi||_{C^0} + \\
 & \;  + ||A-\tilde A ||_{C^0} \cdot c_{A,B}  ||\varphi||_{C^0} + ||B-\tilde B||_{C^0}||\varphi||_{C^0} \\
&\leq  \underbrace{\max\{ \tau \cdot ||A||_{C^0}  d_{A,B},  c_{A,B}, 1 \}}_{=:\tilde d_{A,B}} \cdot(||A-\tilde A ||_{C^0} + ||B-\tilde B ||_{C^0} ||)\cdot || \varphi||_{C^0},
\end{aligned} \]
which proves \eqref{2.1.3}.
\end{proof}
For $t_0 \in \R$ and a  linear equation as in Lemma  \ref{lem21}, consider now the
evolution operators  $U_{A,B}(t_0 + \tau, t_0): C^0([-\tau, 0],\R^n) \to  C^0([-\tau, 0],\R^n) $
which map an initial function  $ \varphi $ to  the solution segment $x_{t_0 + \tau}$, where
$x $ is the solution with $x(t_0 + \cdot)\resr{[- \tau, 0]} = \varphi$.
We briefly write $C^0$ and $C^1$ for $ (C^0([-\tau, 0],\R^n), ||\; ||_{C^0})$ and
$(C^1([-\tau, 0],\R^n), ||\;||_{C^1})  $.
 Since $U_{A,B}$ takes values in $ C^1$, we can consider it as a linear
 operator from $ C^0 $ to $ C^0$ but also from $C^0$ to $ C^1$.

\begin{cor}  \label{evopcont}  One has $U_{A,B}(t_0 + \tau, t_0) \in L_c(C^0, C^1)$ (i.e., the evolution operator is continuous even if one takes the $C^1$ norm in the image), and the maps
\[  \begin{aligned} \, [C^0([t_0, t_0 + \tau], \R^{n \times n}), ||\;||_{L^1}]^2  \ni (A,B) &\mapsto U_{A,B}(t_0 + \tau, t_0)   \in L_c(C^0, C^0) \; \text{ and }  \\
\, [C^0([t_0, t_0 + \tau], \R^{n \times n}),  ||\;||_{C^0}]^2 \ni (A,B) &\mapsto   U_{A,B}(t_0 + \tau, t_0)   \in L_c(C^0, C^1)\;
\end{aligned} \]
	are  locally Lipschitz continuous, with local  Lipschitz constants
	depending only on the values of  $||A||_{C^0}, ||B||_{C^0} $ and $ \tau$.
\end{cor}

\begin{proof} If  $\max\{ || \tilde A - A||_{C^0}, || \tilde B - B||_{C^0}\} \leq 1 $  then
\eqref{2.1.2}  gives for $\varphi \in C^0$  that
\[||U_{A,B}(t_0 + \tau, t_0) \varphi - U_{\tilde A, \tilde B}(t_0 + \tau, t_0) \varphi||_{C^0} \leq
d_{A,B}(||A-\tilde A ||_{L^1} + ||B-\tilde B ||_{L^1})\cdot ||\varphi||_{C^0}, \]
which proves the assertion for the first map.  Combining estimates \eqref{2.1.2} and
\eqref{2.1.3} and again estimating $||\;||_{L^1} $ by  $\tau\cdot ||\; ||_{C^0}$, one sees
that  (in the notation of  Lemma \ref{lem21})
\begin{align*}\; & \quad ||U_{A,B}(t_0 + \tau, t_0) \varphi - U_{\tilde A, \tilde B}(t_0 + \tau, t_0) \varphi||_{C^1}
 =||\Delta||_{C^1} \leq  \\ &  \leq \gamma_{A,B} (||A-\tilde A ||_{C^0} + ||B-\tilde B ||_{C^0})\cdot || \varphi||_{C^0},
\end{align*}
where $\gamma_{A,B} := \max\{\tau d_{A,B}, \,  \tilde d_{A,B}\}$.
The second assertion follows.
 \end{proof}

\begin{rem}
  Statements  similar to the  ones above can be proved for  more general linear equations, but we do not pursue that.
\end{rem}

\medskip
We specialize the above results to matrices given by  the systems from \eqref{3linsyst}, i.e., systems like
\eqref{constcoeff}, but with possibly time-dependent coefficients.
Then we can  consider the spaces $\mathbb{X}$ and $\mathbb{X}^1 $ instead of $C^0([-\tau, 0],  \R^{N+1})$
and   $C^1([-\tau, 0], \R^{N+1})$, and obtain the following result on the evolution operators as defined in \eqref{evop} as immediate consequence of Corollary \ref{evopcont}.

\begin{cor}  \label{evopcontX} The  evolution operators $  U(t_0+ \tau, t_0) $ from \eqref{evop}
are in $L_c(\mathbb{X}, \mathbb{X}^1)$, and  with respect to the norm on this space they
depend locally Lipschitz continuously on the  coefficient functions
$a_i, b_i, c_i$ with the $\max$-norm on $[t_0, t_0 + \tau]$.

The local Lipschitz constant at some specific set of coefficient functions $a_i, b_i, c_i$ can  be expressed in terms of $ \tau $ and $\max_i \max\{||a_i||_{C^0}, ||b_i||_{C^0}, ||c_i||_{C^0}\}$,
where the $C^0$-norms are taken over the interval  $[t_0, t_0 + \tau]$.
\end{cor}

\subsection{Attractor of the semiflow}
We assume  in addition that there exists a radius $R >0$
such that for every initial state $\psi = [\varphi, x_1(0),  \ldots ,x_N(0)]\in \mathbb{X}$ and the corresponding
solution $x^{\psi}_t$  of \eqref{standardfb} there is a time $t(R, \psi)  \geq 0 $  with
\begin{equation}
 || x^{\psi}_t ||_{C^0} \leq R  \text{ for all  }  t \geq t(R, \psi). \label{evbounded}
\end{equation}

This is the case, for example, if  all the $f^i$ in \eqref{standardfb} are of the form
$f^i(u,  x_i, v) = - \mu_i x_i + \tilde f^i(u, x_i, v)   $ (where the argument $u$ is absent in case $i = 0$),
with  $ \mu_i > 0$ and the  functions $\tilde{f^i}$ satisfy $|\tilde{f^i}| \leq M $ for  some  $M > 0$.
Then $ x_i(t)  > (M+ 1)/\mu_i  $ (or $ < (-M- 1)/\mu_i)$)
 implies $\dot x_i(t) <  - 1$ (or $\dot x_i(t) > 1 $), so that  one sees that $ x_i(t) \in [(-M-1)/\mu_i,
(M+ 1)/\mu_i] $ for  all sufficiently large $t$, which implies condition \eqref{evbounded}.
This type of behavior  is  called \textit{point dissipative} in \cite{Hale1988} (Section 3.4, p. 38).
It  follows from the  Arzel\`a-Ascoli theorem that the time-$\tau$-map for the semiflow of system
 \eqref{standardfb} (where $ \tau$ is the delay)  maps bounded sets into compact sets, and together with
property \eqref{evbounded} this implies that the semiflow has a connected, compact global attractor that can be described as
 \[\mathcal{A} := \bigcap_{t \geq 0} \overline{\Phi(t, B)}, \]
where $B \subset \mathbb{X}$ is  the closed ball of radius $R$. (See Theorem 3.4.7, p. 40 and
Theorem 3.4.2, p. 39 in \cite{Hale1988}.)
System \eqref{standardfb} has the  backward uniqueness property:
\begin{equation}  t \geq 0, \psi_1, \psi_2 \in \mathbb{X}, \; \Phi(t, \psi_1) = \Phi(t, \psi_2) \Longrightarrow \psi_1 = \psi_2
\label{backunique}
\end{equation}
This is a consequence of the strict monotonicity conditions in \eqref{monotone}, and proved in \cite{MPS2}, p. 450.
We shall need the following simple  auxiliary result concerning  two-sided flows `contained' in a semiflow.

\begin{prop}  \label{backwflow} Let $X$ be a compact metric space and  $F:[0, \infty) \times X \to X$ a continuous semiflow on $X$.
Assume that all time-$t$-maps $F(t, \cdot) $ are injective. Let $ M\subset  X$ be  a forward invariant set
with the  property that every $ \psi$ in $M$ has a complete  flow line within $M$ in the following  sense:
\begin{equation}\forall \psi \in M \, \exists y^{\psi}: \R \to M: \;   y^{\psi}(0) = \psi, \; \forall t, s \in \R, \, t \geq s: \; y^{\psi}(t) = F(t-s, y^{\psi}(s)).  \label{backw1} \end{equation}
Then the map $F_M: \R \times M \to M,\; F_M(t, \psi) := y^{\psi}(t) $ is  a continuous flow on $M$.
\end{prop}
\begin{proof} The  injectivity condition implies that the  function $y^{\psi}$ is unique for each $ \psi \in M$.
Continuity of $F_M$ on $[0, \infty) \times M$ is clear, since  there $F_M = F$.
 Continuity of $F$ w.r. to the first argument and property \eqref{backw1} imply that all the functions
$y^{\psi}$  are continuous on all of $ \R$. Assume now  $(t, \psi) \in (-\infty, 0] \times M$.
Take $T > |t|$. Then $ F(T, \cdot)$ defines  a homeomorphism  of the  compact space $X$ onto its image
$F(T, X) \subset X$, which set contains $M$, in view of condition \eqref{backw1}.
For $ \tilde t $ in the interval $  (-T, 0]$ (which contains $t$) and  $ \tilde  \psi \in M$ we have
$\tilde t + T \geq 0$ and
\[F_M(\tilde t, \tilde \psi)  =  F[\tilde t -(-T),  F(T, \cdot)^{-1}(\tilde \psi) ] =
    F[\tilde t+T,  F(T, \cdot)^{-1}(\tilde \psi)], \]
which shows continuity of $F_M$ at  $(t, \psi)$.  To show that $F_M$ has the flow property, consider now
$ \psi \in M, t, s \in \R$: If  one of the two numbers, say $t$, is in $[0, \infty)$ then
\[F_M(t, F_M(s, \psi)) = F(t, y^{\psi}(s)) \] and \eqref{backw1} gives
\[F_M(t+s, \psi) = y^{\psi}(s+t) = F(t,  y^{\psi}(s)),\] so that
\begin{equation} F_M(t+s, \psi) = F_M(t, F_M(s, \psi)) . \label{flowM} \end{equation}
If both $t,s$ are in $(-\infty, 0]$ then
\[F[|t+s|, F_M(t+s, \psi)]  = F(|t+s|, y^{\psi}(t+s)) =  y^{\psi}(0) = \psi, \]   and (from the definition of $F_M$):\;
$F_M(t, F_M(s, \psi)) =  y^{y^{\psi}(s)}(t)$, so using $|t + s| = |t| + |s|$ we get
\begin{align*} F[|t+s|,F_M(t, F_M(s, \psi))] & = F[|t+s|,   y^{y^{\psi}(s)}(t)] =
F[|s|, F(|t|,  y^{y^{\psi}(s)}(t))] \\ & =  F[|s|, y^{\psi}(s)] = \psi. \end{align*}
Now injectivity of $F(|t+s|, \cdot)$ gives that again \eqref{flowM} holds.
\end{proof}
For the compact attractor $\mathcal{A} $ of the semiflow $\Phi$ from above, we obtain the following result. It is  known
(compare \cite{MPS2}, pages 449-450), and also \cite{Raugel}, Lemma 2.4, p. 892),   but we include the proof for completeness (reference \cite{Raugel} contains none).
\begin{prop}\label{attrflow}
The attractor $ \mathcal{A}  $ is forward invariant under the semiflow $\Phi$, and
on this attractor $ \Phi$ extends in a unique way to a continuous flow
$\Phi_{\mathcal{A}}: \R \times \mathcal{A} \to \mathcal{A}. $
\end{prop}
\begin{proof}   Continuity of $\Phi(s, \cdot)$ implies
$\Phi(s,  \overline{\Phi(t, B)}) \subset\overline{\Phi(s, \Phi(t, B))} =  \overline{\Phi(s+t, B)}$
for  $ s, t \geq 0$, which shows that  $\mathcal{A}$ is forward invariant:
\[\forall s \geq 0: \;  \Phi(s, \cdot)(\mathcal{A}) = \Phi(s, \bigcap_{t \geq 0}   \overline{\Phi(t, B)}) \subset
\bigcap_{t \geq 0}  \Phi(s, \overline{\Phi(t, B)}) \subset \bigcap_{t \geq 0}  \overline{\Phi(s+t, B)}
\subset\mathcal{A}. \]
Actually we even  have
\begin{equation}   \Phi(s,\mathcal{A}) =  \mathcal{A} \text { for all }  s \geq 0.  \label{invattr} \end{equation}
(this property is called invariance in \cite{Raugel}, Definition 2.3, p. 891).

Proof  of (\ref{invattr}): 1. With the delay time $ \tau$,  we first prove
\begin{equation}    \mathcal{A} \subset \Phi(\tau, \mathcal{A}). \label{tauincl} \end{equation}
If $ \chi \in  \mathcal{A}$ then $\chi \in \overline{ \Phi(2\tau, B)} =
\overline{ \Phi(\tau, \Phi(\tau, B))}$. Hence there exists a sequence
$  (\psi_n) \subset  \Phi(\tau, B) $ with $\Phi( \tau, \psi_n ) \to \chi. $
Since $ \overline{ \Phi(\tau, B)} $ is compact, we can assume that $\psi_n \to \psi^* \in \overline{ \Phi(\tau, B)}$.
Continuity of $ \Phi$ implies that $ \chi = \Phi(\tau, \psi^*) \in  \Phi(\tau, \overline{ \Phi(\tau, B)})
\subset \Phi(\tau, \mathcal{A})$.

2. Inductively, it follows from \eqref{tauincl} that $ \mathcal{A} \subset \Phi(n\tau, \mathcal{A}) $ for $ n \in \N$.
 Using  \eqref{tauincl} and the already proved forward invariance of  $\mathcal{A}$, we obtain
for  $ s \geq 0$ and $  n \in \N$ such that $  n\tau \geq s$:
\[   \Phi(s,\mathcal{A}) \subset  \mathcal{A} \subset \Phi(n\tau, \mathcal{A}) \subset
\Phi(s, \Phi(n\tau - s,  \mathcal{A}) \subset \Phi(s, \mathcal{A}), \]
which proves \eqref{invattr}.

\medskip From  \eqref{invattr} in combination with backward uniqueness, one sees that the semiflow  $ \Phi$ restricted to $[0, \infty) \times  \mathcal{A}$ satisfies the conditions of Proposition \ref{backwflow}
with $X := M := \mathcal{A}$. The conclusion of the present proposition follows.
\end{proof}

\begin{cor}\label{attrsols}  For every   $ \psi = [\varphi,  x_1(0),   \ldots , x_N(0)]$ in the attractor $\mathcal{A}$
there exists a unique solution $x^{\psi}: \R \to \R^{N+1}$ of equation  \eqref{standardfb} with
$x^{\psi}_0 = \psi$. This solution has all states $x^{\psi}_t \;(t \in \R) $  in the attractor,
 and for $n \in \N$, the map
\begin{equation*} 
(\mathcal{A}, || \;  ||_{C^0}) \ni \psi \mapsto  x^{\psi}\resr{[-n\tau, n\tau]} \in (C^1([-n\tau, n\tau], \R^{N+1}), ||\; ||_{C^1})
\end{equation*}
is continuous.
\end{cor}
\begin{proof}
The extension  $\Phi_{\mathcal{A}} $  of the semiflow on the attractor from Proposition \ref{attrflow} provides the solutions $ x^{\psi} $ with states in the attractor.
 The differential delay   equation \eqref{standardfb}  and the continuity of $ \Phi$ (w.r. to $ ||\; ||_{C^0}$) imply that, in particular,  the map
\[(\mathcal{A}, || \;  ||_{C^0}) \ni \chi \mapsto x^{\chi}\resr{[0, T]} \in (C^1([0, T], \R^{N+1}),  ||\; ||_{C^1})\]
is continuous for every $T >0$.
Using Proposition \ref{attrflow} we conclude that the concatenation of maps
\[ (\mathcal{A}, || \;  ||_{C^0}) \ni   \psi \mapsto  \underbrace{\chi :=  \Phi_{\mathcal{A}}(-n\tau, \psi)}_{ \in (\mathcal{A}, || \;  ||_{C^0})}
  \mapsto  x^{\chi}\resr{[0, 2n\tau]} \in  C^1([0, T], \R^{N+1}),  ||\; ||_{C^1}) \]
is continuous, and for $ \psi $ and $\chi$ as above we have
$x^{\psi}(t) = x^{\chi}(t + n\tau) $ for $ t \in [-n\tau, n\tau]$, so the  result follows.
\end{proof}

The attractor can be identified with all bounded solutions which are defined on all of $ \R$, as it was done for the case of scalar equations in \cite{MP88}; see also \cite{Raugel}, Lemma 2.18(d),  p. 899 for a general statement  of this kind, without proof.

\subsection{Discrete Lyapunov functional and eigenspace structure}
 Systems   satisfying  the feedback conditions  from p. 389 of \cite{MPS1}
 are called
\textit{signed cyclic feedback systems} (linear or nonlinear) in \cite{MPS2}. We conclude that  the results of  \cite{MPS1}, in particular concerning the  discrete  Lyapunov functional,  apply to the nonlinear system  \eqref{standardfb} as well as to all  three mentioned linear systems from
\eqref{3linsyst}  (compare the remarks on p. 451 of \cite{MPS2}).
Cases 1) and 3) from \eqref{3linsyst}  will be important for us.

We recall  the discrete (`zero-counting')  Lyapunov functional $V^-: \mathbb{X}\setminus\{0\} \to \{1,3,5,  \ldots , \infty\}$
 (for  the case of negative feedback) from \cite{MPS1}, p. 394: A nonzero  element $\psi = (\varphi, x_1,   \ldots ,x_N) $ of $\mathbb{X}$ can be regarded as a function $x $ defined on $[-\tau,0]\cup\{1,  \ldots , N\}$ by setting $x(\theta) := \varphi(\theta)$ for $\theta \in [-\tau, 0]$ and $x(i) := x_i, i = 1, \dots, N$. Then define the `number of sign changes' of $\psi$
\begin{align*}
\text{ sc}(\psi) := \sup\Big\{k  \in \N_0 & \meig  \exists\;  \theta_0 <   \ldots   <  \theta_k \in [-\tau,0] \cup\{1,  \ldots ,N\}: \\
 & \; x(\theta_{i-1})\cdot x(\theta_i) < 0, i = 1,  \ldots , k\Big\} \nonumber
 \end{align*}
which can be infinity, and is to be read as zero if the set after the $\sup$  is   empty, i.e., $\psi$ has no sign change. The Lyapunov functional is then defined
by
\begin{equation} V^-(\psi) := \left\{
\begin{aligned} \text{ sc}(\psi) \; & \text{ if }  \text{ sc}(\psi)  \text{ is odd or infinite,}\\   \text{ sc}(\psi) + 1 \; &  \text{ if }  \text{ sc}(\psi) \text{ is even }
\end{aligned} \right.
\label{Vdef}
\end{equation}
(see formula (2.6), p. 394 in \cite{MPS1}).
We use the notation $V$  for  $V^-$ from now on, since we only consider the negative feedback case ($\delta^* = -1$ in \eqref{monotone}).  It follows essentially from the definition that
$V$ is lower semicontinuous with respect to the norm defined above on
$\mathbb{X}$, since every sign change  of  $x$ with some $\theta_i$  as in the
definition  of  $\text{sc}$ forces a corresponding sign change for nearby
$\tilde x$ (see \cite{MPS1}, top of p. 395).

Theorem 2.1 from \cite{MPS1}, p. 395  gives the  important result that $V$ is actually a Lyapunov functional for the flow  $ \Phi$ induced by
\eqref{standardfb}, but as well for evolution of the  three types of linear
equations in \eqref{3linsyst}:
Assume $ \psi \in \mathbb{X} \setminus\{0\},\;  t \geq t_0 \geq 0$, and that
$x^{\psi}: [t_0 - \tau, t] \to \R^{N+1}$ is  the  solution  of \eqref{standardfb} or   of one of the  three equations  from \eqref{3linsyst} with initial state $\psi $ at time $t_0$.  Then, for $ t \geq t_0$,   backward uniqueness of
the zero solution (see \eqref{backunique}) implies  $ x^{\psi}_t \neq 0 $, so $ V(x^{\psi}_t)  $ is defined, and
\begin{equation*}  t \geq t_0\;   \Longrightarrow \;  V(x^{\psi}_t)   \leq V(x^{\psi}_{t_0}).   
\end{equation*}

We will be interested in the set
\[\Sigma := \menge{ \psi \in \mathbb{X}\setminus\{0\}}{ V(\psi) = 1}, \]
which  in the present work takes  the part analogous to the set $S$ of nonzero functions with at most one sign change  on  $[-\tau,0]$ from  \cite{Wal91}, p. 81.  Accordingly, solutions $x$ of system  \eqref{standardfb} with their segments $x_t$ in $\Sigma$
(i.e., $V(x_t) = 1$)  correspond to the `slowly oscillating' solutions
of the one-dimensional equation considered in  \cite{Wal91}.
Analogous to formula (6.1) of \cite{Wal91}, p. 81,  we have for  the closure of the set $ \Sigma$ in $\mathbb{X}$:

\begin{prop} $ \overline{\Sigma}  = \Sigma \cup\{0\}.  $ \label{Sigmaclos}
\end{prop}
\begin{proof} It is clear that $ 0 \in  \overline{\Sigma}$, since $  \Sigma \ni  \frac{1}{n} \psi \to 0 \; ( n \to \infty) $ if
$\psi \in \Sigma$.   Conversely, if $ 0 \neq \hat \psi \in  \overline{\Sigma}$ then there exists
a sequence $(\psi_n) \subset \Sigma, \, \psi_n \to \hat \psi$, and lower semicontinuity of $V$ gives
$  1 \leq V( \hat \psi)  \leq \liminf_{n \to \infty} \underbrace{V(\psi_n)}_{=1} = 1,$
so also $\hat \psi \in \Sigma$.
\end{proof}

The set $\Sigma $ is not open in $\mathbb{X}$, but contains subsets which in some sense still
 allow for small perturbation without leaving $\Sigma$.

An important argument from  \cite{Wal91}  (Proposition 6.1,  p. 82) uses the fact that solutions starting in $S$ will have segments without zeroes (which are thus contained in an open subset of $S$, allowing  small perturbation).  In our context this argument can be replaced by using the fact that a solution $  x^{\psi}$  of equation \eqref{standardfb} or of  one of the linear systems  from \eqref{3linsyst}  starting at time $t_0$ with state $ \varphi $ has many so-called `stable times' $t_1 > t_0$, $t_1   \in \text{ Stab}(x^{\varphi}) $ (see \cite{MPS1}, p. 414). (We explain this notion in detail in (2.23) below.)

 Note that the strict feedback condition  for $f^0 $ of Proposition 4.9 in \cite{MPS1} is satisfied in our case, due to the first condition in \eqref{monotone}.
Of course, these  times have to be understood relative to the notion of `solution', i.e., relative to  the nonlinear system or one of the three linear systems  mentioned above.  These stable times
have  the   property that  for all solutions starting at time $t_0$ with
 states $\tilde \psi$ sufficiently close to $\psi$, one also has
\begin{equation}
V(x^{\tilde \psi}_{t_1})  = V(x^{\psi}_{t_1})  \; (= 1, \text{ if } \psi \in \Sigma). \label{stabtime}
\end{equation}
(See \cite{MPS1}, Proposition 4.9, p. 415.)

The proposition below is an easily obtained  consequence of the existence of stable times and will  be useful for us.  The main point is that a kind  of  stable time can be chosen locally uniformly with respect to initial states.

 Let $E(t)$  denote  either the operators $S(t)$ of the semigroup generated by  \eqref{constcoeff} for  $ t \geq 0$,  or the time-$t$-maps $  \Phi(t, 0) $ for the semiflow generated by a system of type \eqref{standardfb} (of course, with the feedback conditions \eqref{monotone}), or $E(t) = U(t, t_0)$, with the evolution operators $ U(t, s)$  for a  non-autonomous system  of type \eqref{nonaut1} to  \eqref{nonaut3}, with a fixed starting time $t_0 $ and $ t \geq t_0$.
All these three types of systems have the form \eqref{constcoeff}, but with possibly time-dependent coefficients.
(An analogous result would also hold for nonlinear systems, but we will  not need that.)

\begin{prop} \label{compactset}  Assume $J \in \N$ is odd, and  let  $K \subset \mathbb{X}\setminus\{0\}$ be a compact subset on
which $V \leq J $. Then there exist
a neighborhood of $ K $  in $(\mathbb{X}, ||\; ||_{C^0})$ and $T_K > t_0 $  such that
\[\forall \psi \in N: \;  V[E(T_K)\psi] \leq J. \]
\end{prop}
\begin{proof} In  the proof we consider the case of time-dependent coefficients, so that
$E(t) = U(t, t_0)$ For every $\varphi \in K $ there exists a stable  time
$t_{1, \varphi} >t_0$ (with respect to the evolution given by the maps $U(t, t_0$))  and an open  neighborhood $N_{\varphi} $ of $ \varphi $  in $\mathbb{X}$ such that, as in \eqref{stabtime},
\[\forall \psi \in  N_{\varphi}: V(U(t_{1, \varphi}, t_0) \psi) =
V(U(t_{1, \varphi}, t_0) \varphi) \leq V(\varphi) \in V(K), \text{ so }  V(U(t_{1, \varphi}, t_0)\psi) \leq  J. \]
Since $K $  is compact,  it is covered by finitely many such  neighborhoods
$N_{\varphi_1},   \ldots , N_{\varphi_k}$. Setting
$N  := \bigcup_{j = 1}^k N_{\varphi_j} $  and
$ T_K := \max_{j = 1,  \ldots ,k}  t_{1, \varphi_j}$,  we obtain for $\psi \in N$
(so  $\psi \in  N_{\varphi_j}$ for some $j$), using that $V(E(t, \psi))$ is non-increasing with $t$:
\[ V(E(T_K)\psi) = V[U(T_K, t_0) \psi] \leq V[U( t_{1, \varphi_j}, t_0)\psi]  \leq J. \]
\end{proof}

Next we quote a combination of results from Theorem 3.2 and
Corollary 3.3  on p. 404 of \cite{MPS1}, specialized to the case of negative feedback (so $V = V^-$). (Part d) contains a very minor extension of the corresponding result from \cite{MPS1}.) By eigenvalues of system \eqref{constcoeff} we mean  the zeroes of the associated characteristic function, compare  formula (3.15), p. 402 in \cite{MPS1}.

\begin{prop} \label{Cor3.3} \begin{enumerate}
\item[a)]  The real parts of the eigenvalues  (of the constant coefficient system
\eqref{constcoeff}, and taking into account multiplicity)  can be ordered as \[\sigma_0 \geq \sigma_1 > \sigma_2 \geq \sigma_3 > \sigma_4 \geq \,   \ldots \]

\item[b)] Define the spaces \[ \tilde{\mathcal G}_{\sigma} := \bigoplus_
{\lambda \mathrm{ \; eigenvalue, \;  } \Real(\lambda) = \sigma} {\mathcal G}_{\lambda},\]  where $ \mathcal{G}_{\lambda} $  denotes   the generalized eigenspace associated to the eigenvalue $ \lambda$.
(We can  think of the complex eigenspaces or their real parts, and then statements about dimension
are to be read accordingly as complex or  real dimension.)
Then $V$ is constant on each `punctured' space   $\tilde{\mathcal G}_{\sigma}\setminus \{0\}$,
 and for every odd $J \in \N$
one has
\[\sum_{\sigma, V = J \mathrm{\;  on \; } \tilde{\mathcal G}_{\sigma}\setminus \{0\} }  \dim  \tilde{\mathcal G}_{\sigma} = 2. \]

\item[c)] The values of $V$ on the spaces from b) (in the real interpretation) are
 \begin{equation} V = \left\{ \begin{aligned}  j + 1 & \text{ on }  \tilde{\mathcal{G}}_{\sigma_j} \text{ if }  j \text{ is even};\\
 j  & \text{ on }  \tilde{\mathcal{G}}_{\sigma_j} \text{ if }  j \text{ is odd}.
\end{aligned} \right.  \label{Vvalues} \end{equation}

\item[d)] Consider a solution $x:[-\tau, \infty) \to \R^n$ of  \eqref{constcoeff},
and assume that there exists $\sigma \in \{\sigma_0, \sigma_1,   \ldots \} $
such that the spectral projection  of $x_0$ to the (spectral) subspace
$  \tilde{\mathcal{G}}_{\sigma} $ of $\mathbb{X}$ is nonzero, and let
$\sigma^{\max}$ be maximal among the $\sigma$ with this property.
Then \[\lim_{t\to  \infty}  V(x_t) \text{ equals the (constant) value of }
 V \text{ on } \tilde{\mathcal{G}}_{\sigma^{\max}}\setminus\{0\}, \]
which  is determined according to  {\textrm{c})}.
\end{enumerate}
\end{prop}

\begin{proof}[\textbf{Comments  concerning  the proof}.]
We  briefly mention important ideas from the proof of  the above result in \cite{MPS1},
  since it is quite essential for the present paper. In particular, we  comment on the precise values of
$V$   (part c) above), the proof of which is not so explicit there.
 Corollary 3.3 from \cite{MPS1} is,  considering the explicit information on the values of     $V^{\pm}$,  actually more a theorem in its own right.
Theorem 3.2  and Corollary 3.3 of \cite{MPS1} are  proved in Section 7 starting on p. 429 of \cite{MPS1}.

\medskip     It is shown  in  \cite{MPS1} that the  quantities  of Theorem 3.2 and Corollary 3.3  remain invariant  under  continuous variation of the constant coefficients  in the linear  system, as long as the sign conditions are preserved. It hence suffices to prove the results  for a special case.
Consider an odd value $ J \in \N$. Define the constants $  \varphi := \frac{\pi}{2(N+1)}$, $\omega := \pi(J- \frac{1}{2})$, $c := \cos(\varphi) $ and $s := \sin(\varphi) $.  Then the functions defined by
$x_j(t) := \cos(\omega t + j\varphi), j = 0,  \ldots ,N$  satisfy the following linear   system in standard (negative) feedback form:
\begin{equation}\left\{\begin{aligned} \dot x_j(t)  &= \omega[-\frac{c}{s} x_j(t) + \frac{1}{s}x_{j+1}(t],\; 0 \leq j \leq N-1, \\
\dot x_N(t)  & = \omega[-\frac{c}{s} x_N(t) - \frac{1}{s}x_0(t-1)].
\end{aligned} \right.   \label{model} \end{equation}
(One main observation in the  verification of \eqref{model} is: Formally defining   \[x_{N+1}(t) := \cos(\omega t +(N+1)\varphi),\]  using the choice of $ \omega $, the facts that $(N+1) \varphi = \pi/2$ and  that $  \cos(u -\pi J) = -\cos(u) $, since $J$ is odd,  one gets that  $-x_0(t-1) = x_{N+1}(t)$, so that  the above equations for $j \in \{0,  \ldots ,N-1\}$
also hold for $j = N$.)
Now for these particular solutions,
$\text{ sc}[x_0\resr{[-1,0]}, x_1(0),   \ldots , x_N(0)] $  equals  the number of sign changes
of $ t \mapsto \cos(\omega t) $ on $ [-1,0]$, since $x_j(0) = \cos(j\varphi) >0, j = 1,  \ldots ,N$.
This sign change number is $J-1$. Thus (compare \eqref{Vdef}) the value of  $V = V^-$  along these  solutions
equals   $(J-1) + 1 = J$, and the dimension of the associated
 linear space of solutions  is two (these solutions  correspond to the eigenvalue pair $\pm  i\omega$ of
system \eqref{model},  compare \cite{MPS1}, p. 437).

\begin{equation}
\begin{aligned} &  \text{\textit{ In particular, the value $J$ is actually assumed by $V$ for this particular system,}}  \\
& \text{\textit{ on a two-dimensional eigenspace (except for  the zero state).} } \label{values}
\end{aligned}
\end{equation}

Now part a) of Theorem 3.2, p. 400 of \cite{MPS1} shows, in particular,  the results listed in italics  under \textbf{1.} and \textbf{2.} below (see also \cite{MPS1}, the section on constant coefficient systems starting on  p. 401):

\medskip \textbf{1.}
\textit{  Consider the spaces  $\tilde{{\mathcal G}}_{\sigma}$,  which are sums of generalized  (real) eigenspaces associated
to eigenvalues with real part $ \sigma$ of system \eqref{constcoeff}. The functional $V$ has a constant  value
$\tilde J(\sigma)$ on $\tilde{{\mathcal G}}_{\sigma}\setminus\{0\}$.
 The direct sum of all spaces  $ \tilde{{\mathcal G}}_{\sigma}$, on which $V$ has a specific value $J$, has exactly dimension two.}

Description of the proof of  \textbf{1.}:
The fact that $V$ is constant on $\tilde{{\mathcal G}}_{\sigma}\setminus\{0\}$   follows from the representation of  solutions $x$ with  states  in $\tilde{{\mathcal G}}_{\sigma}$ as $x(t) = \exp(\sigma t)\cdot q(t) $, with $q$ quasi-periodic For the dimension statements,
see the remark on p. 401 of \cite{MPS1}, in particular, formula (3.12) there.
The fact  the dimension of such a space is at most two is  proved  by showing that in a three-dimensional space  one can construct solutions with a strict  drop in value of $V$ (see  the bottom of p. 431 in \cite{MPS1}). The fact that it actually equals two is proved  in Section 7 of \cite{MPS1}
 by  a homotopy method (in other words,  continuity and hence invariance  of integer quantities w.r. to  variation of the coefficients in  system \eqref{constcoeff}),  considering the  model system
\eqref{model}. It follows then,  of course, that

(i)  all spaces $ \tilde{{\mathcal G}}_{\sigma}$ have at most dimension two
(we call this the dimensional restriction);

(ii) a   direct sum  of different $ \tilde{{\mathcal G}}_{\sigma}$, on which $V$ has the constant odd  value  $J$,   can either  be  only  one two-dimensional space $ \tilde{{\mathcal G}}_{\sigma}$,   or two one-dimensional spaces with distinct $\sigma$.

\medskip \textbf{2.}
\textit{If $x: \R \to \R^{N+1} $  is a solution of a constant coefficient system satisfying condition \eqref{coeffsigns}
(with negative feedback), and $x$   is a finite linear combination
of solutions with nonzero states in some  $\tilde{\mathcal{G}}_{\sigma} $, then
$\lim_{t \to -\infty}\ V(x_t) =  \tilde J( \sigma^{\min}) $ and
$\lim_{t \to +\infty}\ V(x_t) =  \tilde J(\sigma^{\max}) $, where  $ \sigma^{\min}$ and $  \sigma^{\max}$ denote
the minimal and maximal real part from this finite collection, which actually appear in the linear combination representing the solution $x$.}

(See \cite{MPS1}, p. 431 for  the  part of the proof for $ t \to - \infty$; the other part is analogous.)

\bigskip
It is easily seen that  $V = V^-$  takes all odd values on the whole space $ \mathbb{X}\setminus\{0\}$,
simply by choosing initial states with appropriate sign change numbers.
It is not so obvious that $V$ assumes all these  values  even  on the subspaces $\tilde{\mathcal{G}}_{\sigma} $. In  the analogous finite-dimensional situation as considered by
Mallet-Paret and Smith in \cite{MPSmith}, this can be seen in case of even dimension $N = 2b$
directly from the fact that $V$  (in this case, defined by the sign changes between the components in cyclical arrangement)  is clearly bounded above  by $N$ and hence (since $V$ is odd-valued)  by $2b-1$. This  determines
the value of $V = 2h-1 $ on each space in a sequence of $b$ two-dimensional subspaces $\mathcal{G}_{\sigma_{2h-1}} +  \mathcal{G}_{\sigma_{2h}}, \; h = 1,2,  \ldots ,b$ (see formula (2.3) on p. 378 of  \cite{MPSmith}).
 However, even in this  finite dimensional setting,  in the case of  odd dimension $N = 2b+1$
 the determination  of the  precise values of $V$ on the spaces  $\tilde{\mathcal{G}}_{\sigma}$
requires the homotopy method.  These values do not  follow from the dimensional restrictions and the
possible values of $V$ alone.

For the infinite dimensional situation relevant in  the present paper,  the fact that
$V$ assumes prescribed odd  values $J$ on suitable spaces $\tilde{\mathcal{G}}_{\sigma}$ also follows from the homotopy type considerations  in combination with \eqref{values}.
Knowing this, it becomes clear that if the real parts of  the eigenvalues are ordered as
$\sigma_0 \geq \sigma_1 > \sigma_2 \geq \sigma_3 > \sigma_4 \geq \sigma_5\    \ldots $
(the strict inequalities enforced by the dimensional restriction),  then the values of  $V$ must be as described in \eqref{Vvalues}
(compare Corollary 3.3, p. 404 in \cite{MPS1}). This is the only way that $V$ can take all odd values on the spaces $\tilde{\mathcal{G}}_{\sigma_j}$,   nondecreasing with $j$, and  not being constant on a more than two-dimensional sum of these
spaces. So far we gave  explanations concerning the proof of a) to c) in Proposition
\ref{Cor3.3}.

\medskip
The statement in Proposition  \ref{Cor3.3} d)  is a version of the statement for $ t \to \infty$ from  point \textbf{2.} above (compare \cite{MPS1}, formula (3.7) in Theorem 3.1, p. 400),  which allows also an `infinite dimensional part' in  the solution $x$. This part   does not affect the asymptotic value  of $V$.  We indicate the proof: The segments $x_t$ of the solution can be split as
$x_t = (x_t)_{\sigma^{\max}}  + \tilde{(x_t)}$, where
$\tilde{(x_t)}$ satisfies an estimate of the form
\begin{equation}
||\tilde{(x_t)}||_{C^0}  \leq K\cdot \exp(\sigma t), \text{  with }
 \sigma < \sigma^{\max}. \label{Qdecay}
 \end{equation}
This  is the well-known `estimate on the complementary subspace',  see e.g.  \cite{Hal71}, Theorem 22.1, p. 114,
and \cite{HTheory}, Theorem 4.1, p. 181, and \cite{HL}, Theorem 6.1, p. 214. In these  references,
this estimate is  stated  in  a way that does not emphasize the exponential separation. There,  only existence of a  number $ \gamma$  entering the exponent in the estimate is stated (not that it can be chosen small),   and the same exponent $ \beta - \gamma$  (in  \cite{Hal71} and \cite{HTheory}) or
$\beta + \gamma$ (in \cite{HL}), where $\beta $ comes from spectral quantities,   is used for  the estimates on both spaces. The  different growth rates of the semigroup $\{S(t)\}_{t \geq 0}$ (generated by \eqref{constcoeff})    on  spectral subspaces of the generator  with separated real parts  are also  easily obtained from the spectral mapping  theorem for eventually compact semigroups.   See e.g. \cite{EngelNagel}, Corollary 3.12, p. 281.
It follows that the unit vectors $x_t/||x_t||$ in $\mathbb{X}$ converge to  the
(compact) unit sphere  $\mathbf{S}_{\sigma^{\max}}$ in $\tilde{\mathcal{G}}_{\sigma^{\max}}$ as $ t \to \infty$.
Lower semicontinuity of $V$, the fact that $V = J$ on $\mathbf{S}_{\sigma^{\max}}$
(with $J$ given by  \eqref{Vvalues}), and compactness of this sphere imply that there exists a neighborhood  $U$ of  $\mathbf{S}_{\sigma^{\max}}$ in $\mathbb{X}$  with $V \geq J$ on  $U$. It follows that $V(x_t)  = V(x_t/||x_t||_{C^0} ) \geq J $
for  all large enough $t$. On the other hand,  Proposition \ref{compactset}
  shows that there exists a neighborhood $N$ of $\mathbf{S}_{\sigma^{\max}}$ such that for all $ \psi \in N$ one has
$\liminf_{t \to \infty} V[S(t)\psi)] \leq J.$
Altogether we obtain $V(x_t) \to J $ for $ t \to \infty$.
\end{proof}

\subsection{Linearization and local dynamics}
We assumed already  that zero is an equilibrium solution of system \eqref{standardfb}, that the   nonlinearities in \eqref{standardfb} are of class  $C^1$,   and we  consider the linearization at zero. The solutions of the associated characteristic equation are called eigenvalues (they are eigenvalues of the infinitesimal generator of the semigroup on $\mathbb{X}_{\Co}$ which is  induced by the linearized equation.) We make the following assumption, which from now on is always in force:

\medskip \textbf{(A1)} Zero is the only equilibrium of   \eqref{standardfb}, and the  characteristic equation associated to the linearization of  \eqref{standardfb} at zero has a leading  pair  (in the sense of maximal real part)  of complex eigenvalues $ \sigma_0    \pm i\omega$ with  $  \sigma_0, \omega  >0$.

\medskip In the notation  of Proposition \ref{Cor3.3}, assumption \textbf{(A1)} implies that
$\sigma_1 = \sigma_0$, that the space $\tilde{\mathcal{G}}_{\sigma_0}$ is two-dimensional, and that
all eigenvalues except $\sigma_0 \pm  i\omega$ have real parts strictly less than $\sigma_0$. Thus the real parts of the eigenvalues  at zero (compare  part 2 of the comments above)  satisfy
\begin{equation} \label{sigmas}
\sigma_0 = \sigma_1 > \sigma_2 \geq \sigma_3 >   \ldots
\end{equation}
We set $X^{\mathrm{uu}} := \tilde{\mathcal{G}}_{\sigma_0}$ (in the real interpretation) from now on
(the `strongly  unstable' space).
  $X^{\mathrm{uu}}$ then consists of initial states of solutions of the linearized equation which have the
 form
\begin{equation}\R \ni t \mapsto e^{\sigma_0 t}[c_1 (\cos(\omega t) r- \sin(\omega t) s)   + c_2
(\sin (\omega  t) r +  \cos(\omega  t) s)   ], \label{X2form} \end{equation}  where
$r, s \in \R^n$ are linearly independent, and $c_1, c_2 \in \R$.
It follows further from Proposition \ref{Cor3.3} that the space
$X_3 := \tilde{\mathcal G}_{\sigma_2} +  \tilde{\mathcal G}_{\sigma_3}$
(again, in real interpretation) is also two-dimensional (the sum being direct if and only if
$\sigma_2 > \sigma_3$), and that
\begin{equation}V =1 \text{ on } X^{\mathrm{uu}}\setminus\{0\}, \text{ so } X^{\mathrm{uu}}\setminus\{0\} \subset \Sigma,
\text{ and }  V = 3  \text{ on } X_3 \setminus\{0\}. \label{V13values}
\end{equation}

 We have the spectral decomposition
\begin{equation}
\mathbb{X} = X^{\mathrm{uu}}  \oplus X_3 \oplus Q,  \label{specdeco}
\end{equation}
where $Q$  is the spectral subspace of $\mathbb{X}$ corresponding to  all eigenvalues  with real part less or equal to $ \sigma_4$.
(One has $ V \geq 5$ holds on $Q\setminus\{0\}$, but this is not  important for us.)
 We denote by  $\pr  \in L_c( \mathbb{X},  \mathbb{X})$ the spectral projection   which maps $ \mathbb{X}$ onto  the space $X^{\mathrm{uu}}$, defined by the
above decomposition. Analogously, we denote by $\pr_3$ and  $\pr_Q$ the  projections onto $X_3$ and onto $Q$  defined by \eqref{specdeco}.

Note that for the linearization of system \eqref{standardfb}
in case $N >0$, one cannot expect in general an ordering of eigenvalues according to the principle `larger real part implies lower imaginary part', as it is the
case for  $N= 0$; compare the remarks on p. 405, before Section 4  of \cite{MPS1}.
 But this is true for special cases, as  shown e.g. in \cite{BravEtAl}
for   a particular  case of system \eqref{standardfb} (which we will later discuss  as system \eqref{standarduni}), and for the eigenvalues with positive real part.
See the last statement in part (iii) of Lemma 3, p. 4 in  \cite{BravEtAl}.
That is apparently  a consequence of additional restrictions on the coefficients there  ($a_j = 0, c_j < 0$ in the notation from \eqref{constcoeff} of the present work).

\medskip In the lemma below, the statements in a) and b) are analogous  to Remark 6.1, p. 81,  Prop.  6.2, p. 83  in   \cite{Wal91}, and to Lemma 5.1, p. 24 in \cite{Wal95}.

\begin{lem} \label{projonsigma}
\begin{enumerate}
\item[a)]  If $ x^{\psi}$ is a solution of one of the four system types from \eqref{standardfb} or  \eqref{3linsyst}
with $x^{\psi}_{t_0}  = \psi \in \Sigma$ then $x^{\psi}_{t} \in \Sigma $ for $ t \geq t_0$.
\item[b)]  $  \pr(\psi) \neq 0 $ for all $ \psi \in \Sigma$, or, equivalently, $\Sigma \cap (X_3 \oplus Q) = \emptyset$.
\end{enumerate}
\end{lem}

\begin{proof} Ad a): For a solution as in a) and $ t \geq t_0$,  backward uniqueness  of, in particular, the zero solution
(see \eqref{backunique})  implies  that also  $x^{\psi}_{t} \neq 0$. Further, since  $V$ does not increase along solutions
(for all system  types under consideration) and has the minimal value 1,
one sees  that $1 \leq V(x^{\psi}_{t}) \leq V(x^{\psi}_{t_0}) =  V(\psi)  =  1$, hence also $x^{\psi}_t \in \Sigma$.

\medskip Ad b): Let $ \psi \in \Sigma  $ be given and assume    $\pr(\psi) = 0$. Applying Proposition \ref{compactset} to the singleton $ \{\psi\}$ and with $E(t) = S(t)$, we obtain a neighborhood $N$ of $\psi $ in $\mathbb{X}$ and $T >0$  such that
\[ \forall \tilde \psi \in N: \; V[S(T) \tilde \psi] \leq V(\psi) = 1.\]
We can find  $ \tilde \psi  \in N$ such that also
$\pr(\tilde \psi) = 0$, but  $ \pr_3(\tilde \psi)  \neq 0$.
It follows now from  \eqref{V13values}  and  from part d) of Proposition \ref{Cor3.3},  that $\lim_{t \to \infty} V[S(t)\tilde \psi] = 3$, which contradicts $V[S(T) \tilde \psi] \leq 1$ and the monotonicity of $V$.
Thus,   $\pr(\psi) = 0$ is impossible if $ \psi \in \Sigma$.
 (The proof of b)  here is analogous to the proof of Prop.  6.2, p. 83  in   \cite{Wal91}.)
 \end{proof}

\section{A two-dimensional invariant manifold and periodic solutions}

We first consider the local invariant manifold, the global  continuation of which by the semiflow
will be shown to contain a periodic solution in its closure.
With the norm $||\; ||_* $ introduced in the next theorem  and for subspaces $Y \subset \mathbb{X}$, $y \in Y$ and $r >0$
 we shall use  the notation
\[ B^*(y;r; Y):= \menge{z \in Y }{||y-z||_* < r}, \quad \overline{B^*(y;r; Y)}  := \menge{z \in Y }{||y-z||_* \leq r}.\]

Recall that $\sigma_j$ are the real parts of the eigenvalues from \eqref{sigmas}.

\begin{thm} \textbf{(The strong unstable manifold of zero.)} \label{WuuThm}

There  exist $\bar r >0$, a two-dimension\-al local manifold $W^{\mathrm{uu}}(\bar r)$ of class $C^1$, tangent to   $X^{\mathrm{uu}}$ at zero, with the following properties:
\begin{enumerate}
\item[1)]  There exist  an equivalent norm $||\;||_*$  on  $\mathbb{X}$ such that
for $\psi  = \tilde \psi + \psi^{\mathrm{uu}} \in  \mathbb{X}$ with $\tilde \psi \in X_3 \oplus Q,  \psi^{\mathrm{uu}} \in X^{\mathrm{uu}} $
one has $||\psi||_* = \max\{||  \psi^{\mathrm{uu}}||_*, ||  \tilde  \psi ||_*\}$, and a $C^1$-function \\
$w^{\mathrm{uu}}: B^*(0; \bar r; X^{\mathrm{uu}})   \to B^*(0; \bar r; X_3 \oplus Q) $ with
$w^{\mathrm{uu}}(0) = 0, \; Dw^{\mathrm{uu}}(0) = 0$ and
\[ W^{\mathrm{uu}}(\bar r)  = \menge{\psi^{\mathrm{uu}} + w^{\mathrm{uu}}(\psi^{\mathrm{uu}})}{ \psi^{\mathrm{uu}} \in X^{\mathrm{uu}}, ||\psi^{\mathrm{uu}}||_*  < \bar r }. \]

\item[2)]   For $  r \in (0, \bar r)$  we set  $W^{\mathrm{uu}}(r)  :=  W^{\mathrm{uu}}(\bar r)  \cap B_{\mathbb{X}, ||\;||_{*}}(0, r)$; then
\[
W^{\mathrm{uu}}(r) =  \menge{\psi^{\mathrm{uu}} + w^{\mathrm{uu}}(\psi^{\mathrm{uu}})}{ \psi^{\mathrm{uu}} \in X^{\mathrm{uu}}, ||\psi^{\mathrm{uu}}||_*  <  r }.
\]
The manifold $W^{\mathrm{uu}}(\bar r) $ is invariant  in the following sense:
 There exist $r_1 \in (0, \bar r) $,  $ K_1 \geq 1$ and $ \tilde \sigma   >0 $ with $ \sigma_2 <
\tilde \sigma  < \sigma_1 (= \sigma_0)$  such that
  \begin{enumerate} \item[(i)] $ \Phi([0, 1]\times   W^{\mathrm{uu}}(r_1)) \subset  W^{\mathrm{uu}}(\bar r)$;
\item[(ii)] For $\psi \in  W^{\mathrm{uu}}(r_1)$ there exists a solution $x^{\psi}:(-\infty, \infty) \to \R^{N+1} $ of \eqref{standardfb} with initial state $x^{\psi}_0 = \psi$, with $ x^{\psi}_t \in W^{\mathrm{uu}}(\bar r))$  for  all
$ t \in (-\infty, 0]$, and such that
 \begin{equation} ||x^{\psi}_t||_{C^0}  \leq K_1  e^{\tilde \sigma t} \text{ for }  t \in (-\infty, 0].
\label{decay} \end{equation}
\end{enumerate}

\item[3)]For $ \psi \in W^{\mathrm{uu}}(\bar r)$ there exists a solution $x^{\psi}:(-\infty, \infty) \to \R^{N+1} $ of \eqref{standardfb} with $x^{\psi}_0 = \psi$.
There is a  time $T_1$ with  $T_1 < - \tau  < 0$  (uniform for all $ \psi \in W^{\mathrm{uu}}(\bar r)$)
such that  $ x^{\psi}_t \in W^{\mathrm{uu}}(\bar r)$  for  all $ t \in  (-\infty, T_1]$. Also,
with $ \tilde K_1 :=  K_1  e^{-\tilde \sigma T_1}$, one has for such $\psi$
\[||x^{\psi}_t||_{C^0}  \leq \tilde K_1  e^{\tilde \sigma t} \text{ for }  t \in (-\infty, T_1]. \]

\item[4)]  If $ \psi \in  B^*(\bar r; \mathbb{X}) $  has a backward solution with all  $x^{\psi}_t  (t \leq 0) $  in
 $B^*(\bar r; \mathbb{X}) $  and satisfying the analogue of  property \eqref{decay} with some $\hat K_1 > 0 $ and some $\hat \sigma  > \tilde \sigma $,  then $ \psi \in  W^{\mathrm{uu}}(\bar r)$.
\end{enumerate}
\end{thm}
\begin{proof}
(We only indicate the main steps, since there is  a wealth of similar results in the literature which are applicable to this situation. See also the analogue stated  for  a scalar equation in Theorem 5.1 on p. 78 of \cite{Wal91}.)

\medskip Step 1. With $ \sigma_0  $ from \textbf{(A1)}, and $ \sigma_2 $  as in
\eqref{sigmas}, we  choose $\tilde \sigma_2  \in (\sigma_2, \sigma_0) $ (not necessarily positive, if $ \sigma_2 < 0$).
Then an  adapted norm $ || \;  ||_* $ on $\mathbb{X}$  can be chosen  such that the semigroup $\{S(t)\}_{t \geq 0} $ generated by the linearization  of \eqref{standardfb} at zero satisfies
\[ ||S(t) u ||_* \geq e^{ \sigma_0   t} ||u||_*; \quad
 ||S(t) v ||_* \leq e^{\tilde \sigma_2 t} ||v||_* \; \quad (u \in X^{\mathrm{uu}}, v \in X_3 \oplus   Q,  t \geq 0),\] and $||u + v||_* = \max\{||u||_*, ||v||_*\}$.
These estimates, first with additional positive constants as factors,  can be obtained   e.g.  from  \cite{Hal71}, Theorem 22.1, or from  Theorem 2.9, p. 99 of \cite{DiekmannEtAl}, where following  the tradition these estimates  are also stated with the same exponent on both spaces. Just as \eqref{Qdecay} above, they follow
 alternatively from  \cite{EngelNagel}, Corollary 3.12.
  Note also that we can take $ \sigma_0 $ (and not a slightly smaller $\tilde \sigma_0$)  in the exponent for the estimate on $X^{\mathrm{uu}}$,  in view of the representation \eqref{X2form}.  Then, compare  \cite{BJ},   Lemma 2.1, p. 10, for  the adapted norms which  replace the constant  factors by one.

\medskip Step 2. For simplicity of notation, we assume for the rest of the proof that the delay $ \tau $ equals $1$, which is no restriction.  Choose  $\tilde \sigma >0, \tilde \sigma \in (\tilde \sigma_2, \sigma_0)$, and define
$ \rho := e^{\tilde \sigma} (>1)$.
The time  one map $f: = \Phi(1, \cdot) $  associated to system \eqref{standardfb} is of class  $C^1$, and
 $Df(0) = D_2\Phi(1,0) = S(1)$  is  then    $\rho$-pseudo hyperbolic (i.e. $ \sigma(Df(0))\cap \menge{z \in \Co}{|z| =  \rho} = \emptyset $).

Statements analogous to that of Theorem \ref{WuuThm}  for the map $f$ follow from Theorem 5.1, p. 53 and its local version in  Corollary 5.4
on p. 60 of \cite{HPS}, with  $ \bar r >0$ and  local  invariant manifolds   $W^{\mathrm{uu}}_f(r) \; ( 0 < r \leq \bar r$, each of which is
the graph of   the same   function $w^{\mathrm{uu}}$ from  $B_*(0; \bar r; X^{\mathrm{uu}}) $ to  $B^*(0; \bar r; X_3 \oplus Q)$, restricted
to $B_*(0; r; X^{\mathrm{uu}}) $.
 In particular, Theorem 5.1 from  \cite{HPS}  gives  the following dynamical characterization:
$W^{\mathrm{uu}}_f(r)$ coincides with the set of all  $\psi $ in  $B^*(0; r; \mathbb{X})$ with the property   that  there exists a  backward trajectory $(\psi_n)_{n\in -\N_0}  \subset B^*(0; r; \mathbb{X})$   (i.e., $\psi_n = \Phi(1, \psi_{n-1}),\; n \in - \N_0$)  with $ \psi_0 = \psi$,
and with
\begin{equation} ||\psi_n||_{C^0} \cdot \rho^{|n|} \to 0  \text{  as }  n \to - \infty. \label{rhon} \end{equation}
(Here it is not important if one takes $||\;||_{C^0}$ or $||\;||_*$.) It is  not explicitly stated  in Theorem 5.1 from  \cite{HPS}, but can be seen from  the proof (the estimate on
the bottom of p. 59 there, together with the property $ \alpha^{-1} -  \epsilon > \rho $, in the notation of \cite{HPS}) that the boundedness of
$ ||\psi_n||_{C^0}\cdot \rho^{|n|} $ as $ n \to - \infty$, which  for each $\psi_0$ as above  follows from the convergence to zero in \eqref{rhon}, is uniform for all $ \psi \in W^{\mathrm{uu}}_f(\bar r)$, so there exists a constant $M>0$ such that for  all sequences $(\psi_n) $ as above,
 \begin{equation} ||\psi_n||_{C^0}\cdot \rho^{|n|} \leq M \; (n \in -\N_0). \label{Mestim} \end{equation}
The tangency property  $Dw^{\mathrm{uu}}(0) = 0$ follows from Corollary (5.3), p. 60 there.

\medskip Step  3. We simply define     $  W^{\mathrm{uu}}(\bar r) :=     W^{\mathrm{uu}}_f(\bar r)$. The first statement of part 2) of the theorem then follows from the construction of $ ||\; ||_*$ as max-norm w.r. to the spaces $x^{\mathrm{uu}}$ and $ X_3 \oplus Q$.
Continuity of the semiflow implies that we can choose  $r_1 \in (0, \bar r) $  such that
 $ \Phi([0, 1]\times  B^*(0; r_1; \mathbb{X}) \subset B^*(0; \bar r; \mathbb{X})$.

We show that this $r_1$ also has properties  (i) and (ii) above:

Consider $\psi \in  W^{\mathrm{uu}}(r_1)$ and $ \theta  \in [0,1]$.  Then  $\psi$ has a backward trajectory $(\psi_n)_{n \in -\N_0}$ in
$ B^*(r_1; \mathbb{X})$   with  property \eqref{rhon}. Like $ \psi = \psi_0$, all these $ \psi_n$ also satisfy $||\psi_n||_* <  r_1$
and the same dynamical characterization  \eqref{rhon}, hence are contained in
$W^{\mathrm{uu}}_f(r_1)$.
 We can then define a solution $x^{\psi}: \R \to \R^{N+1} $ of \eqref{standardfb}
 by $x^{\psi}_0 := \psi_0 = \psi  $ and
\begin{equation}
x^{\psi}_t = \Phi(t-n,\psi_n) \text{   for }  t \in (n,  n+1),
\label{flowdef}
\end{equation}
for $ n  \in -\N_0$; so $x^{\psi}_n =  \psi_n $ for $ n \in - \N_0$. This  solution    then has all $x^{\psi}_t \;(t \in (-\infty, 1])$  in the set
$ \Phi([0,1] \times  W^{\mathrm{uu}}_f(r_1))$, which is contained  in $\Phi([0, 1]\times  B^*(0; r_1; \mathbb{X}) $ and hence in
$ B^*(0; \bar r; \mathbb{X})$. Therefore, the states $ \chi_n := \Phi(\theta, \psi_n) = x^{\psi}_{n + \theta}$ form a backward trajectory  of the time one map $f$ within $ B^*(0; \bar r; \mathbb{X})$, and $\chi_0 = \Phi(\theta, \psi)$.
Further, there exists a   Lipschitz constant $L$ for  $ (\Phi(\theta, \cdot))$ on  $B^*(0; \bar r; \mathbb{X})$,
uniform for all $ \theta \in [0, 1]$, as follows from  the differentiability properties of the semiflow w.r. to the initial state.
With this $L$,
 \[ ||\chi_n||_{C^0}  \rho^{|n|} = ||\Phi(\theta, \psi_{n}||_{C^0} \rho^{|n|}   \leq L\cdot||\psi_{n}||_{C^0}\rho^{|n|}  \to 0  \; (n \to - \infty).  \]
It follows now from the dynamical characterization of $W^{\mathrm{uu}}_f(\bar r)$  that $ \Phi(\theta, \psi) = \chi_0 \in W^{\mathrm{uu}}_f(\bar r)  =  W^{\mathrm{uu}}(\bar r)$, which proves (i).   Since all $x^{\psi}_t \;(t \leq 0)  $ are  contained in
$\Phi([0, 1]\times   W^{\mathrm{uu}}(r_1))$ (see \eqref{flowdef}),  we now obtain (using  (i)) that   $x^{\psi}_t \in W^{\mathrm{uu}}(\bar r)\;
(t \leq 0)$.

 We prove estimate \eqref{decay}: For $ t \in (-\infty,0]$,  $ t \in [n, n+1] $ with $ n \in -\N$, we get  with $L$ from above and using  \eqref{Mestim}
\[ ||x^{\psi}_t||_{C^0}  = ||\Phi(t-n,\psi_n)||_{C^0} \leq L ||\psi_n||_{C^0} \leq L\frac{M}{\rho^{|n|}}.\]
Recalling the definition of $\rho$,  we conclude (since $t = -|n|+ \theta $ for some $ \theta \in [0, 1]$, and $\tilde\sigma >0$)
\[ ||x^{\psi}_t||_{C^0} \leq LMe^{-\tilde \sigma  |n|} = LM e^{\tilde \sigma  t} \cdot e ^{-\tilde \sigma \theta}
\leq LM e^{\tilde \sigma  t}. \]
This proves \eqref{decay}  with $  K_1 := LM$. (Parts 1) and 2) of the theorem are proved.)

\medskip Step 4.  Proof of part 3) of the theorem: $\psi \in  W^{\mathrm{uu}}(\bar r) $  has a backward trajectory  under $f$  in $W^{\mathrm{uu}}(\bar r)$ satisfying \eqref{rhon}, and also a
corresponding  solution $x^{\psi}$ constructed as in \eqref{flowdef}.
There exists $ N \in -\N, \; N < -\tau$ such that $\psi_n \in W^{\mathrm{uu}}( r_1)  $ for  $n \leq N$. Equivalence of $||\;||_{C^0}$ and $||\; ||_*$ together with estimate (\ref{Mestim}) shows that $N$ can be chosen uniformly for all $\psi \in  W^{\mathrm{uu}}(\bar r)$.
Setting $T_1 := N$, property (i)  from  part 2) of the theorem then shows
 $x^{\psi}_t \in W^{\mathrm{uu}} (\bar r)  $  for  $ t \leq T_1$. Further,
 with $ \chi := x^{\psi}_N, $ we have from (\ref{decay}) for $ t \leq N$:
 $||x^{\psi}_t||_{C^0}  = ||x^{\chi}_{t-N} || \leq K_1e^{\tilde \sigma(t-N)} = K_1 e^{-\tilde \sigma T_1}\cdot e^{\tilde \sigma t}$, which  proves the exponential estimate.

\medskip Proof of part 4) of the theorem:
Assume  $ \psi $ has a backward solution with all  $x^{\psi}_t  (t \leq 0) $ in  $ B^*(\bar r; \mathbb{X})$,   and satisfying the analogue of    \eqref{decay} with some $\hat  K_1 >0 $ and $\hat \sigma > \tilde \sigma$.
Then $\psi_n := x^{\psi}_n \; (n \in - \N_0)$ defines a backward trajectory of $f$ in  $ B^*(\bar r; \mathbb{X})$,
and \[ ||\psi_n ||_{C^0} \rho^{|n|} = ||x^{\psi}_n||\cdot e^{\tilde\sigma |n|}  \leq \hat  K_1  e^{\hat \sigma  n}e^{\tilde\sigma  |n|}=   \hat  K_1  e^{(\hat \sigma  - \tilde \sigma )n} \to 0 \; ( n \to - \infty),\]
so $ \psi_0  = \psi $ satisfies the dynamical characterization and hence is in $  W^{\mathrm{uu}}_f(\bar r) = W^{\mathrm{uu}}(\bar r) $.
\end{proof}

\begin{rem}
  1) With additional  effort it is possible  to arrange that the backward solutions $x^{\psi}$  starting at  $\psi \in  W^{\mathrm{uu}}(\bar r))$
satisfy $x^{\psi}_t \in   W^{\mathrm{uu}}(\bar r)$  for all $ t \leq 0$ (one needs that
$t \mapsto ||\pr(x^{\psi}_t)||_*$ is  increasing on $(-\infty, 0]$), and also to achieve that the last
statement in the theorem holds also  in case $ \hat \alpha = \tilde\alpha$,   but we will  not need this.

2) The phenomenon described  by T. Krisztin for center manifolds in Section 3 of \cite{Krisztin}, namely, the invariance for
a time-$1$ map, but not for the semiflow, is closely connected to the non-uniqueness of center manifolds. It does not occur for the
manifold above, due to the dynamical characterization which implies the uniqueness.
\end{rem}

\subsection{The forward extension $W$ of $W^{\mathrm{uu}}(\bar r)$.}

Combining methods from \cite{MPS1} and \cite{Wal91},  we will study the global forward extension
\[ W := \bigcup_{t \geq 0} \Phi(t, W^{\mathrm{uu}}(\bar r)) \]
of $W^{\mathrm{uu}}(\bar r)$ by the semiflow, which actually defines  a flow on $W$.
The local manifold $W^{\mathrm{uu}}$ and its extension $W$ are analogues of
$W_0 $ and $W$ from p. 79 of \cite{Wal91}.

We want to show that  the closure of $W$ is the union of $W$ with a `bounding' periodic orbit, to which all nonzero solutions in $W$
converge.

The proof of the result  below is somewhat   analogous to the
proof of part 3 of Corollary 6.2, p. 89 in \cite{Wal91}, but here with the set $\Sigma$ in place of the
set $S$ of real-valued continuous functions with at most one sign change from  \cite{Wal91}.

  \begin{lem} \label{pronW}
If $\psi_1, \psi_2 \in \overline{W}$ (the closure of $W$ in $\mathbb{X}$)   and   $\psi_1 \neq \psi_2$ then
$\psi_1 - \psi_2  \in \Sigma$ (see Lemma \ref{projonsigma}), so   $\pr(\psi_1 - \psi_2) \neq 0 $.
In particular, $ \overline{W}\setminus\{0\} \subset \Sigma$.
\end{lem}
\begin{proof} In the proof we  briefly  write $W^{\mathrm{uu}}$ for $W^{\mathrm{uu}}(\bar r)$.

1. Consider first   different states $ \chi_1, \chi_2 \in  W^{\mathrm{uu}}$. We have representations $\chi_j = \varphi_j + w^{\mathrm{uu}}(\varphi_j), \; j = 1,2$,
and \[ \begin{aligned}\chi_1 - \chi_2 & = \varphi_1- \varphi_2 + \int_0^1D w^{\mathrm{uu}}(\varphi_1 + s\cdot
(\varphi_1- \varphi_2))\,ds  (\varphi_1- \varphi_2) \\
& = [\text{ inc}_{X^{\mathrm{uu}}} +  \int_0^1D w^{\mathrm{uu}}(\varphi_1 + s\cdot(\varphi_1- \varphi_2))\,ds]
(\varphi_1- \varphi_2). \end{aligned},  \]
where $\text{ inc}_{X^{\mathrm{uu}}}: X^{\mathrm{uu}} \to \mathbb{X}  $ denotes the inclusion map.
With the  unit sphere $\mathbf{S}^{\mathrm{uu}} $  in $ X^{\mathrm{uu}}$ (w.r.  to $||\;||_*$, but that is not essential), it follows  from  $Dw^{\mathrm{uu}}(0) = 0$ that the unit vectors
$(\chi_1 - \chi_2)/||\chi_1 - \chi_2||_*$ converge to $\mathbf{S}^{\mathrm{uu}} $ as
$\chi_1, \chi_2 \to 0, \chi_j \in W^{\mathrm{uu}}, j = 1,2, \chi_1\neq \chi_2$.

\medskip  2. Consider now  two different  states $\psi_1, \psi_2 $ in $W$. The definition of $W$ and backward uniqueness imply that there exist unique  backward solutions  $x^{\psi_1}, x^{\psi_2}:(-\infty, 0] \to \R^{N+1} $ with $x^{\psi_j}_0 = \psi_j, \, j = 1,2$,
and (from part 2) of Theorem \ref{WuuThm}) that there exists  $T_1 < - \tau < 0  $ such that $x^{\psi_j}_t \in W^{\mathrm{uu}}, j = 1,2 $ for  $ t \leq T_1$, and  the convergence to zero for  $ t \to -\infty$  as described in part 3) of Theorem  \ref{WuuThm} holds for both solutions.
  The first step of the proof implies  that  there exists a sequence $(t_j)_{j \in \N}$
with $ t_{j+ 1} < t_j < T_1  \;(j  \in \N), \; t_j \to - \infty \;(j \to \infty)  $ and a  state $e^* \in \mathbf{S}^{\mathrm{uu}}$ such that

\begin{equation}  (x^{\psi_2}_{t_j}  - x^{\psi_1}_{t_j})/ ||   \ldots  ||_{C^0}  \to e^*
\text{  (in the norm }   ||\; ||_{C^0} \text{  of }   \mathbb{X})
\text{ as }  j \to \infty.  \label{xtnconv} \end{equation}

\medskip 3. The   time-$\tau$-map $ S(\tau) \in L_c(\mathbb{X}, \mathbb{X})$  corresponding to  the
linearization of equation \eqref{standardfb} at zero  and the state  $S(\tau)  e^*$ have  the following properties:

\medskip
(i) There exists a neighborhood  $N_1$ of  $S(\tau)  e^*$ in $(\mathbb{X}^1, ||\;||_{C^1})$ such that $V = 1 $ on $N_1$.

(ii)  $S(\tau)$  can  be regarded as an operator in $L_c[\mathbb{X},
 (\mathbb{X}^1, ||\; ||_{C^1})]$.

\medskip  Property (i) is true  for more general states with a $C^1$ past history, but is particularly obvious for $S(\tau) e^* $ from the representation  of solutions of system \eqref{constcoeff}   with
states in $X^{\mathrm{uu}}$  given in   \eqref{X2form}.

Property (ii)  follows directly from Corollary \eqref{evopcontX}, applied to the special case of a constant coefficient equation. The same corollary also shows that
the evolution operators $ U(t+\tau,t) \; (t \in \R) $  of the system of type \eqref{nonaut1} - \eqref{nonaut3} which describes the development of $  x^{\psi_2}(t) - x^{\psi_1}(t)$,
converge  to $S(\tau)$ in  $L_c(\mathbb{X}, \mathbb{X}^1)$ as $ t \to - \infty$: The
 coefficients $  a_j, b_j, c_j$ of this system satisfy
\[  a_j(t + \cdot) =  \int_0^1 \partial_1f^j[x^{\psi_1}_{j-1}(t)  + s\cdot (x^{\psi_2}_{j-1}(t)- x^{\psi_1}_{j-1})(t), x^{\psi_2}_j(t),  x^{\psi_2}_{j+1}(t)] \, ds
 \to    \partial_1f^j(0),
\] uniformly on $[0, \tau]$ as $ t \to \infty$, with a similar convergence for the $b_j$ and $  c_j$.
The limit values are the constant coefficients of the system that generates $S(\tau)$.
Thus we infer from   Corollary \ref{evopcont} that
\begin{equation} ||U(t+\tau, t) - S(\tau)||_{L_c[\mathbb{X},\mathbb{X}^1]}\to 0 \text{ as } t \to -\infty. \label{Uconv} \end{equation}

\medskip 4. Set $\Delta_t :=  x^{\psi_2}_t  -  x^{\psi_1}_t  \in \mathbb{X}  $ for $ t \leq 0$. From \eqref{xtnconv}  we have
$\Delta_{t_j}/||\Delta_{t_j}||_{C^0}\to e^* \;(j \to \infty)$, and \eqref{Uconv} gives that,  as $j\to \infty$,
\[\Delta_{t_j + \tau}/||\Delta_{t_j}||_{C^0}  = U(t_j+  \tau, t_j) \Delta_{t_j}/||\Delta_{t_j}||_{C^0}
\to S(\tau)e^* \text{ in }   (\mathbb{X}^1, ||\;||_{C^1}). \]
With  $N_1$ from property (i) above we conclude that   for  all large enough $  j \in \N$
one has $ \Delta_{t_j + \tau}/||\Delta_{t_j}||_{C^0}\in N_1$, and hence
\[ V(\Delta_{t_j + \tau}) = V(\Delta_{t_j + \tau}/||\Delta_{t_j}||_{C^0}) = 1,  \]
so $\Delta_{t_j + \tau} \in \Sigma$. The invariance of $ \Sigma$  under the operators $U(t,s) \; ( t \geq s) $  (Lemma \ref{projonsigma} a) now gives that also
\[ \psi_1 - \psi_2 = \Delta_0 =  U(0, t_j + \tau) \Delta(t_j + \tau) \in \Sigma \text{ for these } j,\]
so $ \psi_1 -\psi_2 \in \Sigma$.

\medskip
5. We extend the result of the last step to two different states $ \hat \psi_1,  \, \hat \psi_2 \in
\overline{W}$: For such $ \hat \psi_1,  \, \hat \psi_2$ there exist
sequences $(\psi_{1,n}), (\psi_{2,n}) $ in $W$ with $(\psi_{j,n}) \to \hat \psi_j \, (n \to \infty), j = 1,2$, and we can assume $\psi_{1,n} \neq \psi_{2,n} $  for  all $ n \in \N$. From step 4 we know
$ \psi_{1,n} -\psi_{2,n} \in \Sigma $ for all $ n$, and hence Proposition \ref{Sigmaclos} gives
that also $ \hat \psi_1 - \hat \psi_2 \in \overline{\Sigma} \setminus\{0\} = \Sigma$.
Finally,  Lemma \ref {projonsigma} shows that $  \text{pr}(\hat \psi_1 - \hat \psi_2) \neq 0$.
Now  $\overline W\setminus\{0\} \subset \Sigma$ follows from $0 \in W \subset \overline{W}$, and $\psi = \psi - 0$ for  $\psi \in  \overline W\setminus\{0\}$.
\end{proof}

We now obtain a graph description of $\overline W$. Recall the decomposition from
\eqref{specdeco}.

\begin{cor} \label{graphrep}
\begin{enumerate}
\item[ a)] There exists a unique 
function $w:  X^{\mathrm{uu}} \supset  \pr(\overline{W})  \to X_3 \oplus Q$
 which coincides  with $  w^{\mathrm{uu}} $  on $B^*(0; \bar r; X^{\mathrm{uu}})$ such that
\begin{equation} \overline{W} = \graph(w) := \menge{\psi + w(\psi)}{\psi \in \pr(\overline{W})},
\label{globalgraph}\end{equation}
and $ w $ is continuous.

\item[b)]   $ \pr(W)$ is open in $X^{\mathrm{uu}}$.

\item[c)] Both $\overline{W} $ and $\overline{W}\setminus W$ are  (forward)
invariant under the semiflow $  \Phi$.
\end{enumerate}
\end{cor}

\begin{proof} Ad a):  It follows directly from injectivity  of \text{ pr}  on $\overline{W}$ (Lemma \ref{pronW}) that there exists a unique function $w: \text{ pr} (W) \to  X_3 \oplus Q$ with property  \eqref{globalgraph}.
From the  representation of $W^{\mathrm{uu}}(\bar r) \subset W$   as graph of
$ w^{\mathrm{uu}}: B^*(0; \bar r; X^{\mathrm{uu}})   \to B^*(0; \bar r; X_3 \oplus Q) $  in part 1)  of  Theorem \ref{WuuThm} it is also clear that  $w = w^{\mathrm{uu}} $ on  $B^*(0; \bar r; X^{\mathrm{uu}}) $.
The set $ \overline{W}$ is contained  in the compact attractor $ \mathcal{A}$ and hence compact. Since $\pr $ is injective on $ \overline{W}$, it defines a homeomorphism from $  \overline{W}$ to
$ \pr(\overline{W})$, the inverse of which  is given the map
$ \pr(\overline{W}) \ni \psi \mapsto \psi + w(\psi) $. Thus $\text{ id} + w$ is continuous on
$ \pr(\overline{W})$, and hence also $w$.


\medskip Ad b):  Assume now that $ y \in  \pr(W)$, so there exists $x = \psi^{\mathrm{uu}} + w^{\mathrm{uu}}(\psi^{\mathrm{uu}})  \in W^{\mathrm{uu}} (\bar r) $ and $t \geq 0 $ with
 $y =  \pr(\Phi(t,x))$.
There exists $ \delta > 0$  such that $\overline{B^*(\psi^{\mathrm{uu}}; \delta; X^{\mathrm{uu}})} \subset B^*( 0; \bar r; X^{\mathrm{uu}})$.
The backward uniqueness and injectivity of  \text{ pr} on $W$ imply that the continuous map
\[ h: \overline{B^*(\psi^{\mathrm{uu}}; \delta; X^{\mathrm{uu}})}  \to X^{\mathrm{uu}}, \; \psi \mapsto  \text{ pr } \Phi(t, \psi +  w^{\mathrm{uu}}(\psi))\]
is injective on the compact set $\overline{B^*(\psi^{\mathrm{uu}}; \delta; X^{\mathrm{uu}})}$, hence a homeomorphism onto its image.
Obviously we have $h(\psi^{\mathrm{uu}})  =  \pr(\Phi(t,x)) = y$, and the image of $h$ is contained in
$\pr(W)$. It follows from  Theorem  16C, p. 705 of \cite{Zeidler} (a consequence of the open mapping theorem)
that $h$  maps the interior of $\overline{B^*(\psi^{\mathrm{uu}}; \delta; X^{\mathrm{uu}})}$ in the topology of $X^{\mathrm{uu}}$
  (which is  $ B^*(\psi^{\mathrm{uu}}; \delta; X^{\mathrm{uu}})$) onto the interior of  its image  (again, in the topology of
$X^{\mathrm{uu}}$).  In particular, $y $ is an interior point of this image, which shows it is an interior point of
$\pr W$.

\medskip Ad c) :  Invariance  of $W$ is clear from the definition, which implies invariance of $ \overline{W}$ (we give the
short  argument  for  completeness):
If $  \psi \in \overline{W}$ and $ \theta \geq 0$,
  there exist sequences $(\varphi_n)$ in $W^{\mathrm{uu}}(\bar r)$  and $ (t_n) \subset [0, \infty)$ with $\Phi(t_n, \varphi_n)
 \to \psi$, and then   $ W \ni \Phi(t_n+ \theta,  \varphi_n) = \Phi(\theta, \Phi(t_n, \varphi_n)) \to
\Phi(\theta, \psi)$, so we have also $ \Phi(\theta, \psi) \in  \overline{W}$.
Finally, every  $\varphi \in W$ has a unique backward solution $  x^{\varphi} $ with all  states $x^{\varphi}_t \;(t \leq 0)$  in $W$. If $ \psi \in  \overline{W}\setminus W$ and $  \theta \geq 0$, we know already that $\tilde \psi:= \Phi(\theta,\psi) \in  \overline{W}$. If we had $\tilde \psi \in W$ then
$ \Phi(\theta, x^{\tilde \psi}_{-\theta})= \tilde \psi = \Phi(\theta, \psi)$ would contradict backward uniqueness, since $x^{\tilde \psi}_{-\theta} \in W$ but $\psi \not\in W$.
Hence we obtain that also $\Phi(\theta,\psi) \in  \overline{W}\setminus W$, proving forward invariance of this set.
\end{proof}

\bigskip
Let $W_{\mathrm{per}} $ denote the union of all non-constant periodic orbits in  $\overline{W}$ (we shall see later that there is
only one such orbit).  Then $  \pr(W \cup W_{\mathrm{per}}) \subset \pr(\overline{W})$.
 For every $\psi \in  \pr(W \cup W_{\mathrm{per}}) $ there exists a unique solution $x^{\psi}:\R \to \R^{N+1}  $
of equation \eqref{standardfb} with $  x^{\psi}_0 =   \psi$, and one has
$\psi = \pr(\psi)  + w(\pr(\psi))$.
These solutions define  a continuous flow \[ \tilde \Phi: \R \times  \pr(W \cup W_{\mathrm{per}}) \to \pr(W \cup W_{\mathrm{per}}) \;(\subset X^{\mathrm{uu}}). \]
The formula  \[ v(\varphi) :=  \pr[(\dot x^{\varphi + w(\varphi)})_0] \]
defines a   vector field  $v$ on $ \pr(W \cup W_{\mathrm{per}}) $.

\begin{prop} \label{vfinXuu}
\begin{enumerate}
\item[ a)] The above vector field $v: \pr(W \cup W_{\mathrm{per}})  \to X^{\mathrm{uu}}$ is continuous.

\item[b)]  Let $x:\R \to \R^{N+1} $ denote a solution of \eqref{standardfb}
with states $x_t$ in  $ W \cup W_{\mathrm{per}}$.
Then the curve  $c:\R \to X^{\mathrm{uu}}$  defined by $c(t) :=  \pr(x_t)$   satisfies
\[\dot c(t) = v(c(t)) \; ( t \in \R).\]
\end{enumerate}
\end{prop}
\begin{proof} Ad a): It follows from Corollary \ref{attrsols}  that, in  particular,
the map
\begin{equation}  (W \cup W_{\mathrm{per}},  ||\; ||_{C^0})  \ni \chi \mapsto x^{\chi}\resr{[-\tau, 0] } \in
(C^1([-\tau, 0], \R^{N+1}), ||\; ||_{C^1}) \text{ is continuous.}
\label{C1}\end{equation}
Now for $ \varphi \in   \text{ pr }(W \cup W_{\mathrm{per}}) $ we have with $\chi := \varphi + w( \varphi) = [\chi_0, \chi_1(0),   \ldots , \chi_N(0)] \in \mathbb{X} $ that
$\chi_0 \in C^1([-\tau, 0], \R)$, and
$(\dot x^{\varphi + w(\varphi)})_0 = (\dot x^{\chi})_0 = [\dot\chi_0, \dot \chi_1(0),   \ldots , \dot \chi_N(0)] \in \mathbb{X}$, which in view of    \eqref{C1} and the continuity of $w$ implies that the last expression depends continuously on $ \varphi$. With the continuity of
$\text{ pr}$, the continuity of $v$ follows.

\medskip Ad b): For a solution $x$  and the corresponding  curve $c$ as in assertion b),
the `phase curve' $ t \mapsto x_t \in \mathbb{X}$ is differentiable, with
$ \frac{d}{dt}\resr{t = t_0} x_t = (\dot x)_{t_0}  $ for $t_0 \in \R$.
Hence,  for   $ t_0 \in \R$ we obtain  with
$ \varphi := c(t_0) =   \text{ pr } x_{t_0}, \;\chi := \varphi + w(\varphi) = x_{t_0} $, using the chain rule:
\[\begin{aligned} \dot c(t_0) & = \frac{d}{dt}\resr{t = t_0} \text{ pr } x_t =    \text{ pr} [(\dot x)_{t_0}] =
   \text{ pr} [(\dot x(t_0 + \cdot))_0] =   \text{ pr} [(\dot x^{\chi})_0] =   \text{ pr} [(\dot x^{\varphi + w(\varphi)})_0]\\
& = v(\varphi) = v(c(t_0)). \end{aligned}\]
\end{proof}

\begin{lem} \label{omegaper}
\begin{enumerate}
\item[ a)] For all $\psi \in  W\setminus\{0\}$, the $\omega$-limit set  $\omega(\psi) $ (under the semiflow $ \Phi$)  consists  of a periodic orbit
\[ \mathcal{O}^{\psi} =  \menge{ p^{\psi}_t}{t \in \R} \] which corresponds to a
non-constant  periodic solution $ p: \R \to \R^{N+1} $  of equation  \eqref{standardfb}, and satisfies
$ \mathcal{O}^{\psi} \subset \Sigma \cap (\overline{W}\setminus W$).

\item[ b)] For every periodic orbit $  \mathcal{O} =  \mathcal{O}^{\psi}$ as in a),   the domain of attraction\\
$ \menge{\varphi \in W \setminus\{0\}}{\omega(\varphi) =   \mathcal{O}}$ is open in the  relative topology of
$ W \setminus\{0\}$.
\end{enumerate}

\end{lem}
\begin{proof}  Ad a):  The Poincar\'e-Bendixson theorem from \cite{MPS2} (Theorem 2.1, p. 447) implies that for all $\psi \in  W\setminus\{0\}$, the $\omega$-limit set  $\omega(\psi) $ is either  a single non-constant periodic orbit or contains an equilibrium.  Now, from  assumption (A1), system \eqref{standardfb} has only the zero equilibrium. There exists $r^* \in (0, \bar r]$ and  a  $C^1$-function  $U : W^{\mathrm{uu}}(r^*)  \to \R$ such that
$U(0) = 0$ and
\begin{equation} \forall  \chi \in  W^{\mathrm{uu}}(r^*): \; \frac{d}{dt}\resr{t = 0} U(\Phi(t, \chi)) > 0.
\label{Ufunction} \end{equation}
(With a  suitable quadratic form $q$  on $X^{\mathrm{uu}}, \; U :=  q\circ \text{pr }$ can be taken; compare
Lemma 5.9 and Lemma 5.10 on p. 256 of \cite{DiekmannEtAl}).

 Assume  $0 \in \omega(\psi) $ for some   $ \psi$ in $W\setminus\{0\}$.
Then  $\Phi(t_n, \psi) \to 0$ for some sequence  $(t_n)$ with $t_n \to \infty$,
and (due to backward uniqueness of the zero solution, see \eqref{backunique}):
$\Phi(t_n, \psi) \in W\setminus\{0\} \subset \Sigma$. Since
 $\text{ pr}(W^{\mathrm{uu}}(r^*))  = B^*(0; r^*; X^{\mathrm{uu}})$ is  a neighborhood of zero in $X^{\mathrm{uu}}$ and
 $\text{ pr}$ is injective on $W$,  we  would have
 $\Phi(t_n, \psi) \in W^{\mathrm{uu}}(r^*)$ for all large enough $n$. Also, $U(\Phi(t_n, \psi)) \to U(0) = 0$,
and we  can also find a sequence   $ (s_n)$ of time such that $ s_n \to \infty$ and
 $\Phi(s_n, \psi)  \in W^{\mathrm{uu}}(r^*)$, with  $ \frac{d}{dt}\resr{t = s_n} U[\Phi(t, \psi)] =
\frac{d}{dt}\resr{t = 0} U[\Phi(t, \Phi(s_n,\psi))]
 \leq  0$.
This contradicts  \eqref{Ufunction}.
Hence  $\omega(\psi) $ is a single (non-constant) periodic orbit  $\mathcal{O}^{\psi}$ for all $ \psi \in W\setminus\{0\}$. Obviously $ \mathcal{O}^{\psi} \subset  \overline{W}\setminus\{0\}$, and hence
(see Lemma \ref{pronW})  $\mathcal{O}^{\psi} \subset  \Sigma$. Backward uniqueness
and convergence of all solutions with segments in $W$ to zero for  $ t \to - \infty$
imply that $\mathcal{O}^{\psi} \cap W = \emptyset$, so $\mathcal{O}^{\psi} \subset  \overline{W}\setminus W$.

\medskip Ad b): Consider a periodic orbit $  \mathcal{O} =  \mathcal{O}^{\psi} =  \menge{ p^{\psi}_t}{t \in \R}$  as in a), and let $T >0$ be the minimal period of the periodic solution $ p^{\psi}$.
  Then $\gamma:   [0, T] \ni t \mapsto   \text{ pr } p_t \in X^{\mathrm{uu}} $ defines a   simple closed $C^1$-curve
with  the image $\Gamma = \text{pr } \mathcal{O}$. The curve  $\gamma$ is an integral curve of the vector field $v$ from  Proposition \ref{vfinXuu}. $ \Gamma$  has an interior $\text{ int}(\Gamma)) $  and an exterior $\text{ext}(\Gamma) $ in the  two-dimensional space $X^{\mathrm{uu}}$, in  the sense of the Jordan curve theorem (the two connected components of $X^{\mathrm{uu}}\setminus \Gamma)$). Since   $W$ is connected,  the open set  $\text{ pr }W$ is also connected and    disjoint  to $ \Gamma$, so  $\text{ pr }W$ must be contained  in either
 $\text{ int}(\Gamma) $ or  $\text{ ext}(\Gamma) $.  (Recall that $\text{ pr }$ is injective
on $\overline{W}$, Lemma \ref{pronW}.)
Let $\mathcal{K}(\Gamma) $  be the one of these two sets which contains $\text{ pr }W$.

Continuity of the vector field $v$ and of the flow $\tilde \Phi$ permit us to choose a line segment  $L$  transversal to $\Gamma$ at  $\gamma(0)$,
of the form $  L = \gamma(0) + [0, \varepsilon) \cdot u$ with some $ \varepsilon >0$ and  $ u \in X^{\mathrm{uu}}\setminus \dot \gamma(0)$, and with the following properties:
\begin{enumerate} \item[(i)]   $ L\cap \Gamma = \{\gamma(0) \}$ and $L\setminus\{\gamma(0)\} $ lies  in $\mathcal{K}(\Gamma) $ (the same connected component of $X^{\mathrm{uu}}\setminus \Gamma$ as $\text{ pr }W$);

\item[(ii)] the  vector field $v$ is transversal to $L$ at every  point $\chi \in
 L \cap  \text{ pr }(W \cup W_{\mathrm{per}}) $;

\item[(iii)] there exists a neighborhood $M$ of $ \gamma(0) $  in
$ \text{ pr }(W \cup W_{\mathrm{per}})$ such that for every point $ \chi \in L \cap M$  there exists a unique time $\theta(\chi)$
with a value approximately  equal to  $T$ (the  minimal period of $\mathcal{O}$) such that
$\tilde \Phi(\theta(\chi), \chi) \in L$, and $ \theta(\chi) $ is the first positive time with this property.
\end{enumerate}
(Establishing property (iii) is similar to the usual construction of a return map, except that here it works only for points in
$\text{ pr }(W \cup W_{\mathrm{per}}) $, and we cannot presently know if  all points of $L$ are contained in this set.)

 Recall that $\omega(\psi)=\mathcal{O}$, and let $x^{\psi}: \R\to\R^{N+1}$ be the solution  with initial state $x^{\psi}_0 = \psi$,
 so $x^{\psi}_t  \to \mathcal{O}\;(t \to \infty)$.  Consider the
$C^1$-curve $c^\psi: (-\infty, \infty)  \to X^{\mathrm{uu}} $  given by $ c^{\psi}(t) = \text{ pr } x^{\psi}_t $, which satisfies $c^\psi(t) \to \Gamma $ for $ t \to \infty$ and
$c^\psi(t) \to 0 \in X^{\mathrm{uu}} \; (t \to -\infty)$.
From Proposition \ref {vfinXuu}) $\dot c^{\psi}(t) = v(c^{\psi}(t))$.
An argument known from classical Poincar\'e-Bendixson theory implies that  $c^\psi$ must intersect $L$ in a  sequence of points of the form $ c^{\psi}(t_j) = \gamma(0) + s_j u, \;j = 1,2,3,  \ldots $, where $ t_j < t_{j + 1},\; t_j \to \infty$,  $  s_j  \in (0, \varepsilon) $ for  all $j$,  $(s_j)$  is strictly monotonic,   and  $s_j \to 0 \; ( j \to \infty)$.  Hence we have $s_{j+1} < s_j $.  Since $M$ is a neighborhood of $ \gamma(0) $ in $ \text{ pr }(W \cup W_{\mathrm{per}})$, we can assume that
$ \gamma(0) + [0, s_1] \cdot u \subset  M $.
For every $  j \in \N$, consider   the compact set  $R_j$  in $X^{\mathrm{uu}} $ bounded  by  the curve segment $  c^{\psi}([t_j, t_{j+1}])$, by the line segment $\gamma(0) + [s_{j+1}, , s_{j}]\cdot v $  and by $\Gamma$.
We do  not know at present that $ R_j \subset  \text{pr } \overline{W}$, but the
 set $\tilde R_j := R_j \cap  \text{ pr }(W \cup W_{\mathrm{per}}) $  is forward invariant under the flow  $\tilde\Phi$, and we have
\begin{equation} \tilde R_j \supset \tilde R_{j+1}, \quad  \text{ and (in view of (iii) above) }
L \cap \tilde R_1 \subset  \gamma(0) + [0, s_1]\cdot u \subset M. \label{LinM} \end{equation}

\begin{figure}[htbp]
  \includegraphics[scale = 1.5]{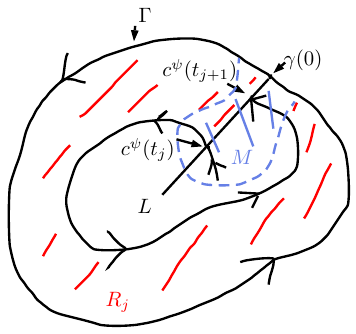}
    \caption{The  argument of classical Poincar\'e-Bendixson-type (the sets $R_j$ and $ M$ indicated)}
     \label{PoincBendFig}
\end{figure}
Now  continuity of the semiflow and  transversality of $v$ to $L$  imply that for nonzero  $\tilde \psi \neq \psi  $  in $W$ sufficiently close to $ \psi$, the corresponding curve $c^{\tilde \psi} $ also has to intersect $L$ at a point  $c^{\tilde \psi} (\tilde t_2) \in M\cap L$ (with $\tilde t_2 $ close to $t_2$), such that
 $c^{\tilde \psi} (\tilde t_2) =   \gamma(0) + \tilde s_{2}v$, where  $s_1 > \tilde s_2  >  s_3$, and $  \tilde s_2 \neq s_2$.
If $  \tilde s_2  > s_2  $ then  $c^{\tilde \psi} (t_2) \in \tilde R_1 $, and if $ \tilde s_2 \leq s_2$ we  even  have $c^{\tilde \psi} (t_2) \in \tilde R_2 $,  in any case $ c^{\tilde \psi} (t_2) \in \tilde R_1 $. Invariance of this set  implies

\begin{equation} c^{\tilde \psi} (t) \in \tilde R_1 \text{  for  all }  t \geq t_2.
\label{invR1}
\end{equation}
Property (iii) of $M$ shows that there exists $\tilde t_3 > \tilde t_2$
(approximately equal to $ \tilde t_2 + T$) with
   $c^{\tilde \psi} ((\tilde t_2, \tilde t_3)) \cap L = \emptyset$
 and   $c^{\tilde \psi} (\tilde t_3) = \gamma(0) + \tilde s_3 \cdot u  \in L$. Since also
$c^{\tilde \psi} (\tilde t_3)  \in \tilde R_1 $, we see from \eqref{LinM}  that we must
 also have $c^{\tilde \psi} (\tilde t_3) \in M$, and $ \tilde s_3 \in [0, s_1]$.
Backward uniqueness for the semiflow $ \Phi$ and injectivity of $\text { pr }$ on
$\overline{W} $ imply that $\tilde s_3 \not \in \{0, s_1, s_2, s_3,  \ldots \} $.
If $ \tilde s_3 \in (s_2, s_1) $ then the construction of $ R_1$ would imply
$c^{\tilde \psi}(t) \not \in R_1 $ for  $ t < \tilde t_3$ close to $t_3$, contradicting
\eqref{invR1}.  Hence we obtain $ \tilde s_3 < s_2$.
Proceeding inductively, analogous arguments give  sequences $(\tilde t_j) $
with $ \tilde t_j < \tilde t_{j+1} \to \infty $  and $(\tilde s_j) \subset (0,s_1)$ such that
 $ c^{\tilde \psi}(t_j) = \gamma(0) + \tilde s_j\cdot u  \in \tilde R_{j-1} \; (j \geq 2)$ and $\tilde s_{j+1} < s_j$.  We conclude  from $s_j \to 0$ that also $ \tilde s_j \to 0$. It follows that $c^{\tilde \psi}(\tilde t_j) \to \gamma(0)$, which implies
$\Phi(\tilde t_j, \tilde \psi) \to p^{\psi}_0 \; (j \to\infty)$. Now the already proved assertion a) shows  that also  $\omega(\tilde \psi) =    \mathcal{O}^{\psi}$ for
$\tilde \psi \in W$ close enough to $ \psi$, which proves b).
\end{proof}

\begin{rem}
  The above argument  showing that the domains of attraction to periodic orbits are open
is similar to the corresponding consideration in part (b) of (24.8), p. 3in \cite{Amann}, where
`one-sided asymptotic stability' of limit cycles is proved. One difference is that in the situation of
Lemma \ref{omegaper} above we do  not know (at this point) that there is a  curve of the projected semiflow through every point  on the `attractive side' of $ \Gamma$. This corresponds to the
distinction between $R_j$ and $\tilde R_j$ above.
\end{rem}

\section {The main theorem}
As a preparation, we need  the following very simple topological fact.

\begin{prop} \label{topprop}
Let $X$ be  a connected topological space and $ U\subset X$ a nonempty and   open set, $U \neq X$. Then, for the boundary of $U$,
one has $ \partial U \neq \emptyset$.
\end{prop}
\begin{proof}  $ \partial U =  \emptyset$ would imply that $U$ is a nonempty,   open and closed proper subset  of $X$, which would then hold also for the complement of $U$ in $X$, contradicting connectedness of $X$.
\end{proof}

Combining the previous results  we can now  obtain a higher-dimensional version of Theorem 10.1 from  \cite{Wal91}.
We assume that  the nonlinearities  in  system \eqref{standardfb} (with negative feedback) are of class $C^1$,   that  zero is the only equilibrium, and that the linearization at zero has a leading  conjugate pair of eigenvalues in the right half plane (i.e., assumption {\textbf{(A1)}} from Section 2 holds). Recall the  corresponding two-dimensional  eigenspace $X^{\mathrm{uu}}$ (a subspace of the state  space $\mathbb{X}$) and the spectral projection
$ \pr $ onto  $X^{\mathrm{uu}}$. Then the following theorem is valid.

 \begin{thm} \label{PerSolMfd}
 \begin{enumerate}
\item[a)] The  two-dimensional invariant manifold  $W\!,$  obtained  by extending the local strong stable manifold   $W^{\mathrm{uu}}(\bar r)$  with  the semiflow,  is globally a graph of a  continuous function $w$ defined on  $ \pr(W)$, which  is an  open subset  of $X^{\mathrm{uu}}$.  With the  closure  $\overline{W} $ of $W$, the set $\overline{W}\setminus\{0\}$   is contained in the set $ \Sigma$ of states in $\mathbb{X}$ where the discrete Lyapunov functional takes the value 1.

\item[b)] There exists a   periodic solution  $p:\R \to \R^{N+1} $ of \eqref{standardfb} such that   the corresponding  orbit $ \mathcal{O} := \menge{p_t}{ t\in \R} \subset \mathbb{X}$ is the $\omega$-limit set
for  all nonzero initial states in   $ W$.

\item[c)] For the `boundary' of $W$ we  have $  \overline{W}\setminus W = \mathcal{O}. $

\item[d)] The projection $\pr(W)$  of $W$ to the two-dimensional space $X^{\mathrm{uu}}$ coincides  with the interior of the simple closed curve $\gamma $ given by $ \gamma(t) := \pr(p_t)$ in the sense of the Jordan curve theorem.
\end{enumerate}
\end{thm}
\begin{proof}
Part a) repeats statements from   Lemma  \ref{pronW} and  Corollary \ref{graphrep}.

\medskip Ad b): From part a) of Lemma \ref{omegaper} we see that all $\psi \in  W\setminus\{0\}$
have a periodic orbit as $\omega$-limit set, and from  part b) of that lemma that  the corresponding domain of attraction $ B_{\psi}$ is open in  $W\setminus\{0\}$. Now connectedness of the latter set  implies that
there can be only one such periodic  orbit, since otherwise, for  arbitrary $\psi \in    W\setminus\{0\}$,
\[  W\setminus\{0\} = B_{\psi} \;  \cup  \bigcup_{ \tilde \psi \in  W\setminus\{0\}, \tilde \psi \neq \psi } B_{\tilde\psi}\] would be a decomposition into nonempty, open, disjoint subsets.

\medskip Ad c): It is clear that $\mathcal{O} \subset  \overline{W}$, and $ \mathcal{O} \cap W = \emptyset$
follows from backward uniqueness, since   the solution $x^{\psi}$ satisfies $x^{\psi}_t \to 0 \;(t \to - \infty) $ for   all $ \psi \in W$. So we have $\mathcal{O} \subset \overline{W}\setminus W$.

Proof of $ \overline{W}\setminus W \subset  \mathcal{O}$: With $ \bar r$ from Theorem \ref{WuuThm}, consider  the circle
\[ \tilde S := \menge{\varphi \in X^{\mathrm{uu}}} { ||\varphi||_* = \bar r/2} \text{  in }  X^{\mathrm{uu}}, \]
and the corresponding set $S := \menge{\varphi + w(\varphi) }{\varphi \in \tilde S} $  in $W$.
Assume now $ \chi \in \overline{W}\setminus W $ and $ \varepsilon >0$.  There exist  sequences
$(\psi_n) \subset W\setminus \{0\}$ and $(t_n) \subset [0, \infty)$ such that $\Phi(t_n, \psi_n) \to \chi$. Since every flow line in $W\setminus \{0\}$ passes  through $S$ and $ S$ is compact,
 we can assume that $(\psi_n) \subset  S$  and $ \psi_n \to \psi^* \in S$.
Then necessarily $t_n $ has no bounded subsequence $(s_n)$, since otherwise
we can assume $ s_n \to s^* \in [0, \infty)$, and then $ \chi = \lim_{n \to \infty}
\Phi(s_n, \psi_n) = \Phi(s^*, \psi*) \in W$, contradicting to $\chi\not\in W$.
Thus we have $t_n \to \infty$.  Now the compact set $S$ is contained in the domain of attraction
of $ \mathcal{O}$, which implies that there exists $T>0$ such that for all
$\psi \in S$ and all $ t \geq T$, $\dist(\Phi(t, \psi),\mathcal{O}) < \varepsilon. $
Hence $\dist(\Phi(t_n, \psi_n),\mathcal{O}) <   \varepsilon$ for all sufficiently large $n$, implying that
$\dist(\chi,\mathcal{O})\leq \varepsilon$. Since this is true for every
positive $\varepsilon$, we conclude $\chi \in   \mathcal{O}$.

\medskip Ad d):  As in the proof of part b) of Lemma \ref{omegaper}, define
$ \Gamma$  as the image of $ \gamma$, so $\Gamma = \pr(\mathcal{O})$. With
$ pW := \pr(W) \subset X^{\mathrm{uu}}$,   part c) of the present theorem shows that
$\pr(\overline{W}) = \pr(W \cup \Gamma) =  pW \cup \Gamma $ (in disjoint union).
Continuity of $ \pr$ and $  \mathcal{O} \subset \overline{W}$ imply that
$ \Gamma \subset \overline{pW}$. (The last closure refers to the closure  in $X^{\mathrm{uu}}$, in principle, but this is the same as the closure in $ \mathbb{X}$.)
Since $\Gamma $ is disjoint to the open subset   $ pW$ of $X^ {uu}$, we conclude
$\Gamma \subset \partial(pW) $, where the boundary symbol $ \partial$ again refers to the topology of $X^{\mathrm{uu}}$. Further, compactness of   $\overline{W}$ and continuity of $\pr$ imply
\[ \overline{pW} =  \overline{\pr(W)}\subset \overline{\pr(\overline{W})} = \pr(\overline{W}) =  pW \cup \Gamma,\]
from which we see that $ \partial(pW)  \subset \Gamma$, so together we have
\[ \Gamma =  \partial(pW). \]

Now, since $pW$  is connected (as continuous image of the connected set $W$) and disjoint to $ \Gamma$,
$pW$ must either be contained in the interior $\text{ int}(\Gamma)$ or in  the exterior
$\text{ ext}(\Gamma)$ (the two connected components of $ X^{\mathrm{uu}} \setminus \Gamma$ in the sense of the Jordan curve theorem, see e.g. Theorem 5.4, p. 362 in \cite{Dug}). Let $K$ be the one  of these two  sets containing $pW$.
If we had $pW \neq K$, then Proposition \ref{topprop} would imply $\partial_K pW \neq \emptyset$ for the
boundary of $pW$ in the relative topology of $K$. But points in this boundary would also be points of the boundary $\partial pW$ of $pW$ in $X^{\mathrm{uu}}$, since $pW$ is also open in $X^{\mathrm{uu}}$. But we know
that $\partial(pW) = \Gamma$ is disjoint to  $K$, a contradiction.  Hence we obtain $pW = K$.
As  $pW$ is bounded and $\text{ ext}(\Gamma)$ is unbounded, we must have
$pW = \text{ int}(\Gamma) = K$.
\end{proof}

\begin{rem}
  Apart from the geometric description  of  the  periodic orbit as `boundary'  of the  manifold $W$, the above theorem has  the advantage
that the  only spectral assumption is an unstable  leading conjugate pair.
We briefly compare this  to  previous  results:
\end{rem}

1) Theorem \ref{PerSolMfd} allows coexistence  of  real  eiqenvalues (which had to be excluded in the papers  \cite{IvaBLW04} and \cite{IvaBLW20}).  We show in Section 5 that, for  the special case of  system  (\ref{standardfb}) with unidirectional  coupling,  this situation can occur precisely if the system dimension is  three or larger.

2) The  method from \cite{IvaBLW20} required a lower bound for  the spectral projection  to the leading unstable space, restricted to a cone
$\mathfrak{K} $ (Condition 3) of Theorem  3.4 on p. 5391 there), where the cone contains initial  states leading to slowly oscillating solutions.   This condition
means practically exclusion of a second conjugate pair of unstable eigenvalues  with imaginary part in $(-\pi/\tau, \pi/\tau)$, if $ \tau $ is the delay.
It was shown in \cite{BravEtAl}, again for the unidirectional coupling case, that  such a  second pair can exist precisely if the system dimension is  5 or larger. See also the corresponding remarks on p. 5392 of
\cite{IvaBLW20}. Again, such a second pair corresponding to  slow oscillation is
allowable in Theorem \ref{PerSolMfd}.

Of course, on the other hand, the   above  theorem
requires  the  additional   assumption  of monotonous feedback.

\section{Attractor location, stability border  and oscillation border  for unidirectional coupling}
\label{locstabosc}
In  this  section we aim to  complement Theorem \ref{PerSolMfd} with information in two directions:
\begin{enumerate}
\item[1)] Estimates on the location of the attractor (and thus, in particular, on the manifold $W$ and its bounding periodic orbit from Theorem   \ref{PerSolMfd});
\item[2)] Investigation of two important numbers associated to the overall feedback strength around the loop  of a cyclic system, which will be called the `oscillation border'  $K_c$ and the `stability border'  $K_u$.
\end{enumerate}
For both purposes, we focus on a simpler special case of system \eqref{standardfb}, namely
systems  of the form
\begin{equation}
\left\{\begin{aligned}
\dot x_1(t)   & = - \mu_1 x_1(t)  + g_1(x_2(t)), \; \\
 \ldots  &  \ldots    \ldots \\
\dot x_j(t)   & = -\mu_j x_j(t) + g_j(x_{j+1}(t)),\; 1 \leq j \leq N-1, \\
 \ldots  &  \ldots    \ldots \\
\dot x_N (t)  & =  - \mu_N x_N(t) + g_N(x_{1}(t- \tau)).
\end{aligned} \right. \label{standarduni}
\end{equation}
where
\begin{equation} \label{fbassum}
\mu_j \geq 0, \;  g_j(0) = 0  \text{ for } j = 1,  \ldots ,N,
\text{ and }  g_j' > 0 \text{ for } j \in \{ 1,  \ldots ,N-1\}, \text{ while }  g_N' < 0.  \end{equation}

(This corresponds to system \eqref{standardfb} with  only $N$ instead of $ N+1$ state variables, which are numbered from $1$ to $N$,
with $ g_j$ not dependent on $x_{j-1}$, and with negative feedback.)
The reasons for this  restriction are  that  the techniques for application of interval maps
seem to be not developed for the more general coupling  structure of system
\eqref{standardfb}, that a detailed  analysis of the characteristic  equation  of
system \eqref{standarduni} is available, and that many systems of interest in the natural sciences fall in the class covered  by \eqref{standarduni} (after suitable transformation).

In analogy to previous sections we take the state space
\[\mathbb{Y} :=  \menge{[\varphi, x_2(0),   \ldots , x_N(0)]}{ \varphi \in C^0([-\tau, 0],\R) , x_j(0) \in \R, \; j = 2,  \ldots N}\]
for  system \eqref{standarduni}.
We include the  simple result below for completeness. We note that the proof uses only the property of the $g_j$
that   $g_j^{-1}(\{0\}) = \{0\}$, and that  the overall feedback around the loop is negative
(but the monotonicity properties are not necessary).

\begin{prop} \label{uniqueq} Under assumptions \eqref{fbassum}, system \eqref{standarduni}
has the  zero solution as the only  equilibrium.
\end{prop}
\begin{proof}
Obviously zero is an equilibrium.
Assume now that  $(x_1^*,   \ldots , x_N^*) $ is an equilibrium.
System \eqref{standarduni}  can be written as
$\dot x_j(t)    = -\mu_j x_j(t) + g_j(x_{j+1}(t)),\; 1 \leq j \leq N,$ with the index $j+1$ taken $\mod N$.  From  the cyclic structure  and the fact that every $g_j$ has $0$ as its only zero, one obtains  the implication
\begin{equation} \label{allzero}
x_j^* = 0 \text{ for   some } j \Longrightarrow  x^*_{(j + 1)\text{mod} N} = 0  \Longrightarrow  x_k^*  = 0 \text{ for all } k.     \end{equation}

If we had one $x_j^* \neq 0$, and then, in  view of \eqref{allzero},
all $  x_k^* $ nonzero, then all terms $g_j(x_{j+1}^*)$ would have to   be nonzero,  showing that necessarily all $\mu_j $ are nonzero.
It follows then  that
\[ x_1^* = \frac{1}{\mu_1} g_1(x_2^*)= \text{(if } N \geq 2) =
 \frac{1}{\mu_1} g_1( \frac{1}{\mu_2} g_2(x_3^*)) =   \ldots
 =\frac{1}{\mu_1}g_1 (\frac{1}{\mu_2} g_2(   \ldots  \frac{1}{\mu_{N}} g_N(x_1^*)   \ldots ).\]
From  the feedback assumptions,  the composition on the right hand side has negative feedback and hence $0$ as its only fixed point, so necessarily $x_1^* = 0$, contradicting  the property that all $x_j^*$ are nonzero.  Thus, zero is the only  equilibrium.
\end{proof}

\subsection{Approximate location of the attractor and of the manifold $W$ from Theorem \ref{PerSolMfd}.}

\noindent
Based on the  method that goes  back to  Ivanov and Sharkovsky \cite{IvaSharko}, one can describe invariant  and attractive sets for  system  \eqref{standarduni}, which then contain  the attractor, and thus in particular  the manifold $W$ together with  the periodic orbit in the closure of $W$.  Such a description was given in Theorem 5.2 on p. 5394 of \cite{IvaBLW20}. With a view on applications in Section 6,  we give  a slightly modified version of  that theorem which is true if we assume that $g_N$ in \eqref{standarduni} is bounded and not only bounded to one side (above or below).
 As in \cite{IvaBLW20}, we  assume that $ \mu_j>0$ for all $j$, and set
\[
G_j := \frac{1}{\mu_j}g_j, \; j = 1,  \ldots,  N, \text{ and } G := G_1 \circ    \ldots  \circ G_N.
\]

\begin{thm} \label{attrloc}
Under assumptions \eqref{fbassum}, assume in addition that all $ \mu_j$ are positive and  that $G_N$ is  bounded. Define the compact intervals
\[ I_1  := \overline{G(\R)}, \; I_N := G_N(I_1), \; \text{ and }   I_j := G_j(I_{j+1}),
j \in \{ 2,  \ldots ,N-1\}.\]
Then the set\newline
$\tilde{\mathbb{Y}} :=  \menge{[\varphi, x_2(0),   \ldots , x_N(0)] \in \mathbb{Y}}
{ \varphi ([-\tau, 0]) \subset I_1,  x_j(0) \in I_j\; j = 2,  \ldots,  N}$
\newline is attracting and invariant
for the semiflow on $\mathbb{Y} $ induced by equation \eqref{standarduni}.
In  particular, this set contains the closure $\overline{W}$ of the two-dimensional manifold $W$  as described in Theorem \ref{PerSolMfd}.
\end{thm}
\begin{proof} From Proposition \ref{uniqueq} we know that   system
 \eqref{standarduni} has only zero as equilibrium.
Under  the further assumptions here, $ I_1 = \overline{G(\R)} $ is already a  compact and $G$-invariant interval, which contains zero in its interior (recall the conditions on
$g_j'$). Hence we can  proceed with the proof as in the proof of part 2 of  Theorem 5.2 in
  \cite{IvaBLW20}, using this  interval $I_1$ instead.
\end{proof}
We take the opportunity to correct  minor  errata   from  the just quoted paper
\cite{IvaBLW20}. The first of these  has to do with  differences between the version of \cite{IvaBLW20}  on ArXiv and the commercially  published version.

\bigskip
1. In the beginning of the proof of Proposition 5.1 on p. 5394 of  \cite{IvaBLW20}, it should be:

   \ldots  contained in [4]  as the claim `$x_j(t) \in F_j(A_{j+1}), \, t \geq s_j$' within the proof of Lemma 6   \ldots

\medskip 2.  In reference [4], it should be: Discrete Contin. Dyn. Syst.,  Ser. S 13 (1)

\bigskip

Note also  that the argument in the proof of Lemma 6 from \cite{BravEtAl}, and hence also the argument of Proposition 5.1 from
\cite{IvaBLW20},  both use the property that zero   is the only  equilibrium of the system; otherwise only weaker statements like\newline
`$ \liminf_{t \to \infty} \text{dist}(x_j(t), I_j) = 0$' instead of `$x_j(t) \in I_j $ eventually' would be true.

\subsection{Oscillation border  and   stability border}

\noindent  We continue with  the  characteristic equation  associated to the zero solution of \eqref{standarduni}, which   was investigated in detail in a number of papers, in particular, in  \cite{BravEtAl}. (Compare also, e.g., \cite{Belair1992} for a related  equation with more general coupling but all decay coefficients equal, \cite{IvaBLW04} for the case $N = 3$ and
\cite{IvaBLW20} for general $N$.)
Define $a_j := g_j'(0), j = 1,  \ldots ,N$, so $a_j >0, j \in \{1,   \ldots,  N-1\}$, and $a_N <0$,
and set
\[
K := |\prod_{j = 1}^N a_j | = |\prod_{j = 1}^Ng_j'(0)| >0.
\]
The characteristic function $\chi$ and the characteristic equation associated to the zero equilibrium of \eqref{standarduni} are then
given by

\begin{equation}
\chi(\lambda) := (\lambda + \mu_1)\cdot   \ldots  \cdot(\lambda + \mu_N)  + Ke^{-\lambda\tau } = 0.
\label{CharEq}
\end{equation}
The zeroes of $ \chi$ are  also called eigenvalues (they are eigenvalues of the generator of the semigroup defined by the linear equation), and  determine the stability of the zero  solution.

Note that  only the family $(\mu_1,   \ldots ,\mu_n)$  of decay coefficients, the `loop delay'
$ \tau$ (which occurs as sum of  all delays, when similar systems with several delays are transformed to
the shape (\ref{standarduni}))  and the product
of the local feedback strengths $a_j$ at zero enter the characteristic function. This fact  reflects
the  unidirectional cyclic coupling structure of the system.
We consider equation \eqref{CharEq}  from the point of view of varying the parameter $K$.
With the application of Theorem \ref{PerSolMfd} in
 mind,  two quantities are of particular interest, which we call the
`oscillation border'  $K_c\geq 0 $ (the index $c$ stands for `complex'), and the `stability border' $K_u > 0 $. We will show below that there exist
unique such numbers with the following properties:
\[
K \in [0,  K_c] \Longleftrightarrow \text{ Equ.   \eqref{CharEq} has  real solutions;}
\]
so for $K > K_c$  all eigensolutions  of the linearization of  \eqref{standarduni}
are oscillatory. Further,
\[
K> K_u  \Longleftrightarrow \text{ Equ.  \eqref{CharEq} has  solutions with positive real part}.
\]
In case $ K> K_u$ the  system always has a  leading pair of  (properly) complex conjugate eigenvalues
with maximal real part among all eigenvalues, so that Theorem 4.2 applies.

We  are interested in the question  which of the numbers $K_c$ and $K_u$ is smaller;
it was observed in \cite{IvaBLW04} that for  $N = 3$  both inequalities
$K_c < K_u $ and $K_c > K_u$ are possible.
In order to obtain periodic solutions  via fixed point methods  in the papers \cite{IvaBLW04} ($N = 3$),   \cite{IvaBLW20} (general $N$)
and also in  \cite{IvaBLW23} and \cite{HT} (for $N= 1$), one had to exclude real eigenvalues,
in order to provide generally oscillatory solution behavior.

On the other hand, Theorem  \ref{PerSolMfd} from the present  work is applicable
whenever one has a leading unstable complex pair of eigenvalues, and exclusion
of simultaneously existing real eigenvalues (which are automatically  negative,  as one
sees directly from \eqref{CharEq}) is unnecessary.
Therefore the possible cases where  $K_c > K_u$ (instability then  coexists with real negative eigenvalues for $K \in (K_u, K_c]$ are of interest. In these cases the results of
Theorem \ref{PerSolMfd}  are true and give in particular the periodic solution in the closure of the manifold $W$, while  results  in the spirit of   \cite{IvaBLW20} do not apply.
Note however, that  on the other hand Theorem \ref{PerSolMfd}  requires the monotonicity
 assumptions which are not needed in \cite{IvaBLW20}.

\medskip Regarding  the oscillation  and   stability borders, we state the result below.
The case $N = 2$ is a particular challenge to treat because, as
$ \mu_j \to 0, j = 1,2$,  both numbers  $K_u$ and $ K_c$ also necessarily approach zero.

\begin{thm} \label{thm52}
For statements a) and b) below, assume that $ \mu_1,   \ldots , \mu_N > 0 $
and $ \tau >0$ are given.
\begin{enumerate}
\item[a)]  There exist unique $K_c \geq 0 $ and $ K_u>0$  as described above, and for  $K > K_u$
the characteristic equation has a unique leading complex conjugate  pair of solutions in the
right half plane, as required in Theorem \ref{PerSolMfd}.

\item[b)] $K_u$ is given by
\begin{equation} K_u = \prod_{j = 1}^N \sqrt{\omega_1^2 + \mu_j^2},
\label{Kueq}\end{equation}
 where
$\omega_1 \in (0, \pi/\tau) $ is the unique solution to the equation
\begin{equation}\pi-\omega\tau  = \sum_{j = 1}^N \arctan(\omega/\mu_j).
\label{omegaeq}\end{equation}

\item[c)]  Depending on the dimension $N$, the possibilities for inequalities between these numbers are
as follows:  For $N \in \{1,2\}$,  one always has $K_c < K_u$,
while for $ N \geq 3$ both inequalities $  K_c < K_u$ and $  K_c >  K_u$ are possible
for suitable choices of $\tau$ and the $\mu_j$.
 \end{enumerate}
\end{thm}
\begin{proof} For the statements in a) and b) concerning $K_u$, see Lemma 3 in
\cite{BravEtAl}. (The number $\omega_1$  is the solution of equation (10) in
\cite{BravEtAl} for the  case $k= 1$.)

Obviously, equation \eqref{CharEq} never has positive real solutions.
Define  the polynomial function  \[p(x) := (x + \mu_1) \cdot   \ldots  \cdot (x + \mu_N);\]
 then \eqref{CharEq}
has a negative real solution for some $K \geq 0$ if and only if there exists  $x < 0 $ such that
$ 0 \geq  -Ke^{- x\tau } \geq p(x)$ (certainly for $ K = 0$),
and then this property  also holds for all $\tilde K  \in [0, K]$.
Now the estimate   $|p(x)| \leq \max\{|x|, \max_j\mu_j\}^N$ shows that for large enough $K>0$, there are no real zeros of $\chi$.
Thus defining \[K_c := \sup\menge{K \geq 0}{\text{ equ. }  \eqref{CharEq} \text{ has a real solution}}\]
one obtains,  by using  continuity,  that   the characteristic equation has real (negative) solutions if
and only if $ K \in [0,  K_c]$.  (See also Lemma 1 in  \cite{BravEtAl}.)  Parts a) and b) are proved.

\medskip Proof of c):
\textit{The  case $N = 1$:} Writing $\mu$ instead of $\mu_1$,  equation  \eqref{CharEq}  becomes
 $\lambda + \mu  = - K e^{-\lambda \tau}$, and formula \eqref{Kueq}  gives
$K_u = \sqrt{\omega_1^2 + \mu^2}$, with  $ \pi - \omega_1\tau  = \arctan(\omega_1/\mu) < \pi/2$, so
$\omega_1 > \pi/(2\tau)$ and hence $ K_u >  \pi/(2\tau)$.
Writing $z = \tau(\lambda  + \mu)$ and $ \tilde K := \tau e^{\mu \tau} K$,
the characteristic equation is   equivalent to
$ z = -\tilde K e^{-z}$. It is known since, e.g.,  the work of Wright (\cite{Wright55}, Theorem 5, p. 72) that
$\tilde K_c := e^{-1}$   is the oscillation border for this equation,
corresponding to $ K_c = \frac{e^{-1}}{\tau e ^{\mu \tau}}.$
Thus we have
$ K_c <  \frac{e^{-1}}{\tau} <  \frac{\pi}{2\tau} < K_u$.

\medskip \textit{The case $N = 2$:} The characteristic equation now is
\begin{equation}
\chi(\lambda) = (\lambda + \mu_1)(\lambda + \mu_2)  +   Ke^{-\lambda\tau } = 0.
\label{N21}
\end{equation}
We  can assume  (without restricting generality) that  $ \mu_1 \geq \mu_2 >0$.
At the oscillation border $K_c$, one must have $ \chi(\lambda) = 0$ and also
 $ \chi'(\lambda) = 0$ for some real negative $ \lambda$, so in addition to \eqref{N21}
(with $K = K_c$) one has
\begin{equation}
2\lambda + \mu_1+ \mu_2  -  K_c\tau e^{-\lambda\tau } =0.
\label{N22}
\end{equation}
Then the polynomial $p(z) := (z + \mu_1) (z + \mu_2) $ must  have a non-negative derivative and a non-positive value at $ \lambda$,
so that  with $ \bar\mu := (\mu_1 + \mu_2)/2$ one has
\begin{equation} \lambda  \in [-\bar \mu, -\mu_2].
\label{N23}
\end{equation}
From \eqref{N21} and  \eqref{N22}  one obtains
\[(\lambda + \mu_1)(\lambda + \mu_2)  = -\frac{2}{\tau }\lambda -\frac{2\bar \mu}{\tau},\]
or $   \lambda^2 + 2(\bar \mu + \frac{1}{\tau}) + \frac{2\bar \mu}{\tau} + \mu_1 \mu_2 = 0$,
hence
$ \lambda = -(\bar \mu + \frac{1}{\tau}) \pm \sqrt{(\bar \mu + \frac{1}{\tau})^2 -
 \frac{2\bar \mu}{\tau} -  \mu_1 \mu_2 }. $
The expression under the square root  equals
 \[ \frac{\mu_1^2 + \mu_2^2}{4} + \frac{2\mu_1\mu_2}{4}  +  \frac{2\bar \mu}{\tau} + \frac{1}{\tau^2}
 -  \frac{2\bar \mu}{\tau} -  \mu_1 \mu_2  = \left(\frac{\mu_1 - \mu_2}{2}\right)^2 + \frac{1}{\tau^2}, \] in particular it is always positive.  From \eqref{N23} we see that
$\lambda \geq - \bar \mu$, so we must have
\begin{equation}
\lambda = -(\bar \mu + \frac{1}{\tau}) + \sqrt{ \left(\frac{\mu_1 - \mu_2}{2}\right)^2 + \frac{1}{\tau^2}}, \label{lamb}
\end{equation}
from which (recall $ \mu_1 \geq \mu_2$)  we obtain the estimate
\[
\lambda \leq \underbrace{-\frac{\mu_1}{2} - \frac{1}{\tau} +  \sqrt{ \frac{\mu_1^2}{4} + \frac{1}{\tau^2}}}_{  =:\bar\lambda(\mu_1) }.
\]
Note that  $ {\bar \lambda}'(\mu_1) = - \frac{1}{2} + \frac{1}{2\sqrt{  \ldots }}\cdot \frac{\mu_1}{2}  < -\frac{1}{2} + \frac{1}{2\cdot \mu_1/2}\cdot \frac{\mu_1}{2} = 0,$  so
 $\bar \lambda$ decreases with increasing $ \mu_1$.
  From \eqref{N22} and \eqref{lamb} we can now conclude
\begin{equation} K_c = \frac{e^{\lambda \tau}}{\tau}2(\lambda + \bar \mu) =
 \frac{2e^{\lambda \tau}}{\tau} \left( - \frac{1}{\tau} + \sqrt{ \left(\frac{\mu_1 - \mu_2}{2}\right)^2 + \frac{1}{\tau^2}}\right).
\label{N24}
\end{equation}
We prepare an  inequality concerning $K_c$: Since for all $z\in \R$ one has  $  p(z) \geq  p( -\bar \mu) = - \frac{(\mu_1 - \mu_2)^2}{4}$, equation \eqref{N21}  and $ \lambda<0$ imply
\begin{equation} K_c < \frac{(\mu_1 - \mu_2)^2}{4} \leq \frac{\mu_1^2}{4}.
\label{N26}
\end{equation}
Recall the number $\omega_1 \in (0, \pi/\tau)$ from part b) of the present theorem,  satisfying
\begin{equation}
\pi - \omega\tau = \arctan(\omega/\mu_1) + \arctan(\omega/\mu_2), \text{ and }
K_u =  \sqrt{\omega_1^2 + \mu_1^2}\cdot \sqrt{\omega_1^2 + \mu_2^2}.
\label{N27}
\end{equation}
We distinguish several cases now:

\textit{Case 1:} $ \omega_1 \geq \mu_1.  $ Then, in view of \eqref{Kueq} and  \eqref{N26},
\[ K_u \geq  \sqrt{2\mu_1^2 }\cdot \omega_1  = \sqrt{2}\mu_1 \omega_1 \geq \sqrt{2} \mu_1^2> \frac{\mu_1^2}{4} > K_c. \]

\textit{Case 2:} $ \omega_1 < \mu_1$.
Then the first equation in \eqref{N27} shows that $  \pi-\omega_1\tau  < \arctan(1) + \pi/2 = 3\pi/4$,  so
\begin{equation} \pi/4 < \omega_1 \tau < \mu_1 \tau.  \label{omega1esti}\end{equation}

\textit{Subcase 2a:} $ \mu_1 < \frac{\pi}{2\tau}$. Then
$  \frac{\pi}{4\tau} < \omega_1 < \mu_1 < \frac{\pi}{2\tau}.$
It follows that
\[ K_u \geq   \sqrt{ 2\omega_1^2} \cdot\sqrt{\omega_1^2 } \geq \sqrt{2} \frac{\pi}{4\tau} \cdot  \frac{\pi}{4\tau}
= \sqrt{2}  \frac{\pi^2}{16\tau^2},\]
and  \eqref{N26} gives
$ K_c \leq  \frac{\mu_1^2}{4} \leq  \frac{\pi^2}{16\tau^2} <  \sqrt{2}  \frac{\pi^2}{16\tau^2} \leq K_u. $

\medskip \textit{Subcase 2b:} $ \mu_1 \geq \frac{\pi}{2\tau}$. Then, since $\bar\lambda$ is a decreasing function of $ \mu_1$,  $\lambda$  from \eqref{lamb}
satisfies
\[\lambda \leq\bar \lambda(\mu_1) \leq \bar\lambda( \frac{\pi}{2\tau})
=  -\frac{\pi}{4\tau} - \frac{1}{\tau} +  \sqrt{ \frac{\pi^2/(4\tau^2)}{4} + \frac{1}{\tau^2}}
=  \frac{1}{\tau}[(-(1 + \pi/4) +  \sqrt{  (\pi/4)^2 + 1}],
\]
 and hence \eqref{N24} gives
\[K_c  \leq  \frac{1}{\tau} 2 \exp\left[-(1 + \pi/4) +  \sqrt{  (\pi/4)^2 + 1} \right] \cdot
\left( - \frac{1}{\tau} + \sqrt{ \left(\frac{\mu_1 - \mu_2}{2}\right)^2 + \frac{1}{\tau^2}}\right).
\]
We use the numerically computed estimate
$ 2 \exp\left[-(1 + \pi/4) +  \sqrt{  (\pi/4)^2 + 1} \right] \leq 1.19  \ldots  \leq \frac{6}{5} $ to conclude
\[ K_c \leq   \frac{1}{\tau}\cdot  \frac{6}{5}  \left( - \frac{1}{\tau} + \sqrt{\frac{\mu_1^2}{4}  + \frac{1}{\tau^2}}\right) = \frac{1}{\tau^2} \frac{6}{5}\left(\sqrt{\frac{\mu_1^2\tau^2}{4} + 1} -1 \right).
\]
For $K_u$ we now have (in view of \eqref{omega1esti})
\[K_u \geq\sqrt{\omega_1^2 + \mu_1^2}\cdot \omega_1 \geq
\sqrt{\left(\frac{\pi}{4\tau}\right)^2  + \mu_1^2} \cdot \frac{\pi}{4\tau}
= \frac{\pi}{4\tau^2}\sqrt{\frac{\pi^2}{16} + \mu_1^2\tau^2}.
   \]
With  the  above estimate for $K_c$, and setting $u := \mu_1^2 \tau^2 $ (so $ u \geq \pi^2/4$ in the current subcase 2b)),
the assertion $K_c < K_u$ follows if  we show that
\[ \frac{6}{5} \left(\sqrt{\frac{u}{4} + 1 } - 1\right) <
\frac{\pi}{4}\sqrt{u + \frac{\pi^2}{16} } \text{ for all } u \geq \pi^2/4,  \text{ or } \]
\begin{equation}
\underbrace{\sqrt{u + 4 } - 2}_{ =: \varphi(u)} \;  < \;
\underbrace{\frac{5\pi}{12}\sqrt{u + \frac{\pi^2}{16} }}_{ =:\psi(u)} \quad \text{ for all }\quad u \geq \pi^2/4. \label{N28}
\end{equation}
Obviously $0 = \varphi(0) < \psi(0)$, so it suffices to prove $\varphi'(u) \leq \psi'(u) $ for all $ u \geq 0$.
The latter inequality  is equivalent to $ \frac{1}{2\sqrt{u+4}}  \leq \frac{5\pi}{12}\cdot \frac{1}{2\sqrt{u + \pi^2/16}}$,
and hence to   $ \frac{u+4}{u + \pi^2/16}\geq \left(\frac{12}{5\pi}\right)^2$, or
\[
\left(\frac{12}{5\pi}\right)^2 u + \frac{9}{25} \leq u + 4, \text{which is true  for all }  u \geq 0.
\]
Thus  \eqref {N28} is true (even for all $ u \geq 0$),  and $K_c < K_u$ is proved also in subcase 2b).

\bigskip
\textit{The case $N \ge 3$}: (Here we can extend the corresponding arguments  from \cite{IvaBLW04} for $ N=3$; in particular, the fact
that now in equation \eqref{omegaeq} there are at least three $\arctan$-terms allowing  to have $ K_u \to 0 $ while $ K_c$  remains
bounded from below (unlike the case $N = 2$).

\noindent \textit{Proof that  $  K_c < K_u$ is possible:}  Choose $M \geq 1$ such that
$|z|^N < e^{-z}$ for all $ z \in (-\infty, -M]$. Then, for $z \leq -M$ and $K \geq 1$, one has
\[( z + M)^N + Ke^{-z}  > -|z|^N + e^{-z} \geq 0, \]
so that the function $z \mapsto (z + M) ^N +  Ke^{-z} $ has no zeroes in $ (-\infty, -M)$ and therefore also no zeroes in $(-\infty, 0]$. This is the characteristic function from
\eqref{CharEq} if we set $  \tau = 1$ and all $ \mu_j $ equal to $M$, so we can conclude
$K_c < 1$ for such parameters. From formula \eqref{Kueq} we see that
in this case $ K_u > \prod_{j= 1}^N \mu_j = M^N \geq 1$, so $K_c  < 1 \leq K_u$.

\medskip \textit{Proof that  $  K_c > K_u$ is possible:}
Here we distinguish the cases of even and odd $N$.
We set $ q(z) := z^N$ if $ N \geq 3 $ is odd, and $q(z) := (z+1)(z+4)\cdot z^{N-2}$ if $N \geq 3 $ is even.

In both cases, the function $q$ is negative on the interval $ I := [-3,-2]$, and (by continuity) there exist $m_1 >0, K_1 >0$ such that
one has
\begin{align*} \chi(z) &:= (z+1) (z+4)\prod_{j = 3}^N(z+ \mu_j) + K e^{-z}  < 0 \\ & \quad  \text{ for } z \in I,\;  \mu_j \in (0, m_1] \;( j \geq 3), \;  K \in (0, K_1]  \text{ if } N  \text{ is even};  \\
\chi(z) &:=  \prod_{j = 1}^N(z + \mu_j) +   K e^{-z}  < 0   \text{ for } z \in I, \; \mu_j  \in (0, m_1] \\
& \quad \text{ for all } j,   \; K \in (0, K_1]  \text{ if } N  \text{ is odd}.
\end{align*}
(The reader can notice that the left hand sides of the above inequalities are the characteristic functions for $ \tau = 1$ and the corresponding choice of the $ \mu_j$.)
In both cases,  since the exponential term makes $\chi(z)$  go to $+\infty$ as $ z \to - \infty$, the function $ \chi$ has a zero in $(-\infty, -3)$.
This means that  if we set $ \tau := 1,  \mu_1 := 1; \mu_2 := 4 $ and take $\mu_j  \in (0, m_1] $ for $ j \geq 3 $ in case $N $ is even, and if we take $ \tau := 1$ and  \textit{all}    $\mu_j \in  (0, m_1] $ in case $ N$  is odd,
then
 \begin{equation}
 K_c \geq K_1\quad  \text{is satisfied in both cases.} \label{N29}
 \end{equation}
\textit{In case of even $ N$,} one has a choice of at least two values $ j \in \{3,   \ldots ,N\}$. Letting $\mu_j \to 0 $ for these $ j$, but keeping fixed $  \mu_1 = 1, \mu_2 = 4$,  one sees from formula \eqref{omegaeq}
that $\omega_1 \to 0$ as $ \mu_j \to 0 \; (j \geq 3)$, since otherwise at least two
$\arctan$-terms in that formula with $ \mu_j \to 0$ already give $ \pi $ in the limit (recall that we set $ \tau :=1$).  Formula \eqref{Kueq} now shows that $K_u \to 0 $ as  $ \mu_j \to 0 \; (j \geq 3)$,
so using \eqref{N29} we obtain $K_u < K_1 \leq K_c$ for $ \tau = 1,    \mu_1 = 1, \mu_2 = 4$, and all sufficiently small $\mu_j, \;  j \geq 3$.

\medskip\noindent\textit{In case of  odd $ N \geq 3$},   we obtain as above that $\omega_1 \to 0 $ as
\textit{all} $ \mu_j \to 0$, and again formula \eqref{Kueq} shows that  $K_u \to 0 $ in this case.
Thus, we obtain again $K_u <  K_1 \leq K_c$ for $ \tau = 1 $ and all $ \mu_j $ sufficiently small.
\end{proof}

\section{Cyclic gene regulatory networks}
In this final section we indicate applications of our results to systems used as models for cyclic gene regulatory networks --
in particular,  the famous `repressilator' of Elowitz and Leibler.  Compare
the  well-known paper by Goodwin \cite{Goodwin65}, and
\cite{BanMah78a}, \cite{BanMah78b}, 
\cite{EloLeib}, \cite{Mah80}, \cite{TakadaEtAl2012}, \cite{HoriEtAl2013} (some  references with typical   character for this broad field).
These systems are of the form
\begin{equation}
\left\{\begin{aligned}
\dot r_i (t)  & = -a_i r_i (t) + \beta_i f_i (p_{i-1} (t -  \tau _{p_{i-1}})),\; i = 1,  \ldots ,n, \\
\dot p_i(t)  & = -b_i p_i (t) + c_i r_i(t - \tau_{r_i}), \;  i = 1,  \ldots ,n.
\end{aligned} \right. \label{genereg}
\end{equation} Here the index $i$ is to be taken  mod $n$, i.e., $0$ is to be read as $n$,
and  we use  lower indices for the nonlinearities here.
The constants $a_i,  \beta_i, b_i, c_i $ are positive constants, and the delays satisfy
$ \tau_{p_{i-1}} \geq 0 $ and    $ \tau_{r_i} \geq 0. $
 Note that the coupling here is to the previous neighbor in the sense of the index $i$, as opposed to the `next' neighbor coupling in system \eqref{standarduni}.
 The  monotone functions $ f_i$   are  so-called Hill functions:
\begin{equation} f_i(x)  = \frac{1}{1 + x^{\nu_i}_i} \text{ (decreasing), or }
f_i(x)  = \frac{ x^{\nu_i}}{1 + x^{\nu_i}}  \text{ (increasing)},
 \label{deacinc}
\end{equation}
with exponents  $\nu_i \geq 1$,  frequently chosen in the interval $[2,3]$.
 The variable $r_i$ represents  the  concentration of the $i$-th  sort of messenger RNA, and $p_i$ the concentrations of an associated protein that influences the transcription (creation) of the next RNA type in the chain.   In order to have negative feedback around the whole loop, it is necessary that the total number of indices $i$ with decreasing $f_i$ is odd.
See \cite{TakadaEtAl2012}, formulas (1) and (2) on p. 3.   The special case  of \eqref{genereg} with $n = 3$ and all three $ f_i$ of the decreasing type   is known as  model for  the actual biological system called  `repressilator',
 which was `synthetically' created in  bacteria in \cite{EloLeib}. The original differential equation model in \cite{EloLeib} had no delays.

The total number of state variables in \eqref{genereg}  is obviously $N: = 2n$.
It is indicated in detail how system \eqref{genereg}  can be transformed into the  form
\eqref{standarduni} with $ \tau = 1$ in Proposition 1, p. 4 of
\cite{TakadaEtAl2012} and in Appendix A of the same paper.

In particular:
\begin{enumerate}
\item[1)] The decay coefficients $ \mu_j \; (j = 1,  \ldots ,N)$
  of  the resulting system, which has the form  \eqref{standarduni},  satisfy
\begin{equation}\mu_{2j-1} = Tb_{N+1-j}, \; \mu_{2j} =  Ta_{N+1-j}, \; j = 1,  \ldots , n.
\label{mutrans}
\end{equation}
with the  $a_i$   and $ b_i$ of system \eqref{genereg},
and  the total loop delay $ T  = \sum_{i =1}^n (\tau_{p_{i-1}} + \tau_{r_i})$ of the latter.

\item[2)] The `nonlinearities'  $g_{2j-1}, \; j = 1,  \ldots ,n$ in system \eqref{standarduni}  are actually linear and given by
\begin{equation*}g_{2j-1}(x) = T\cdot c_{N+1-j}x, j = 1,  \ldots ,n.
\end{equation*}

\item[3)]  The nonlinearities  $g_{2j}, j = 1,  \ldots ,n$ in system \eqref{standarduni}
are given by
\begin{equation*}g_{2j}(x) = \sigma_{2j}T \cdot \beta_{N+1-j} f_{N+1-j}(\sigma_{2j+1}x), j = 1,  \ldots ,n,
\end{equation*}
 with the coefficients $\beta_i$ from  \eqref{genereg} and with
appropriately chosen signs $ \sigma_k  \in \{-1, +1\}$ for $  k = 1,  \ldots , 2n+1$,
to achieve that   $g_j' > 0,  j \in \{2, 4,  \ldots ,N-2\}$, and $  g_N' < 0$.
\end{enumerate}

We will show below that  system  \eqref{genereg}  has a unique equilibrium  (in $(0, \infty)^{2n}$)
which is denoted by $(r_1^*, p_1^*,   \ldots , r_n^*, p_n^*)$, and  the corresponding  transformed system of type \eqref{standarduni}
also has a unique equilibrium $(x_1^*,   \ldots , x_N^*)$ (in the region
in $\R^{2n}$  obtained by transforming   $(0, \infty)^{2n}$),
but it will not be zero. One sees from formula (6) on p. 4 of \cite{TakadaEtAl2012}
that   $ x_{2j-1}^* = \pm p_{N+j-1}^*,  \;  x_{2j}^* = \pm r_{N+1-j}^*, j = 1,  \ldots ,N. $
We obtain the following natural  relation between quantities at the equilibria:
Let  $K:= |\prod_{j = 1}^Ng_j'(x_{j+1}^*)| $ (the index taken mod $N$).
Note that, due to negative loop feedback, the product without absolute value is negative.
The positive  number  can be seen as the total feedback strength of the nonlinearities around the loop at   the equilibrium.  Then
\begin{equation}
K = T^n \cdot  \left(\prod_{i = 1}^nc_i\right )\cdot T^n \left(\prod_{i = 1}^n\beta_if_i'(p_{i-1}^*)\right) =
T^{2n} \left(\prod_{i = 1}^nc_i\right ) \cdot \left(\prod_{i = 1}^n\beta_if_i'(p_{i-1}^*)\right), \quad \label{Ktrans}
\end{equation}
where the index $i$ is taken by modulus $n$.

In order to obtain a system  of type  \eqref{standarduni} with zero as the unique equilibrium
(assuming that the original gene regulatory system  \eqref{genereg}  has a unique equilibrium),
one has to apply an  additional transformation
replacing  $ x_j$  in  \eqref{standarduni} by $y_j := x_j - x_j^*$.
The resulting system in the $y_j$  then has the same decay coefficients $\mu_j$, the same delay  $\tau = 1$, and  the same  total
feedback strength $K$ at the equilibrium (which is then  zero).
These quantities are known from above,  which is why we do not write this transformation explicitly. Note, however,  that  while the
existence of a  unique equilibrium  for  \eqref{genereg}  is not too difficult to see  from the monotonicity of the $f_i$
(see Lemma \ref{lem61} below and compare, e.g., with \cite{HoriEtAl2010}), its explicit  values are in general  accessible only
numerically.

The biologically given invariance of the set of initial conditions  with
non-negative values for system \eqref{genereg} leads to a corresponding
 invariant set for the final equation  of form \eqref{standarduni}   with zero as unique
equilibrium -- clearly this will not be the set of non-negative states.

\subsection{Attractor location for system \eqref{genereg}}

\noindent Similar to Theorem \ref{attrloc}, we want to roughly describe the location of the
global attractor for system \eqref{genereg}. Instead of going through the transformations
and then applying Theorem \ref{attrloc}, it seems preferable  to directly state the following simple consequences of the  original system \eqref{genereg}.

\medskip In this section  below we always assume initial  states with all values $ r_i(t)$  and $p_i(t)$
(also with delayed arguments) in  $[0,\infty)$. By the state of a solution at time $t\geq 0$ we mean the restriction of
$(r_i(t+ \cdot), p_i(t+ \cdot)), i = 1,  \ldots ,n$
to the `loop delay  interval' $[-T, 0]$, although most variables will appear with smaller delay  than $T$ in  \eqref{genereg}.

 Also, we use  the terms `invariant' and `attracting'  in the sense of forward (in time) invariant/attracting.
 The content of this lemma is known; we provide its proof for the sake of completeness.

\begin{lem}
\label{lem61}
\begin{enumerate}
\item[a)] The set of states where  $r_i  \in (0, \beta_i/a_i) $
and $ p_i \in (0, c_i/b_i), \; i = 1,  \ldots ,n$  is forward  invariant and  attracting   for system \eqref{genereg}.

\item[b)] There exists a unique  equilibrium state $(r_1^*, p_1^*,   \ldots , r_n^*, p_n^*)$ of system
   \eqref{genereg}, which (of course) is  contained in the set described in a).
\end{enumerate}
\end{lem}
\begin{proof}
From the variation of constants formulas
\[ \begin{aligned} r_i(t) &  = \exp(-a_i t)r_i(0)  + \beta_i\int_0^t\exp(-a_i(t-s)) f_i[p_{i-1}(s - \tau_{p_{i-1}})] \, ds, \\
 p_i(t) &  = \exp(-b_i t)p_i(0) + c_i \int_0^t\exp(-b_i(t-s))  r_i(s - \tau_{r_i}) \, ds,
\end{aligned}
 \]
and from $f_i \geq 0$, one sees that the set of states with values in $[0, \infty) $
and also the set of states with values in $(0, \infty)$ are both (forward) invariant.
 Further, at least one $f_i$, say $f_{i_0}$, is of the decreasing  type and hence strictly positive, which shows that  $r_{i_0}$ becomes
 (and then remains) strictly positive.
 It follows that $p_{i_0} $ becomes (and stays) strictly positive, and, going inductively around the loop, system \eqref{genereg} shows
 that  this is eventually true for all variables.
Thus, the set of strictly positive states is attracting (for  the set of initial states with non-negative values)  and forward invariant.

Further, we see from \eqref{genereg} that if all variables are positive on $[t-\tau, t]$  and  $r_i(t) \geq \beta_i/a_i$ then,  $\dot r_i(t) < 0$.
This implies, at first,   that the set of states with $r_i \in (0, \beta_i/a_i) $ is  invariant, and must contain any equilibrium state.
Further, if a solution would satisfy, $r_i(t)  \geq \beta_i/a_i$ for all $ t\geq 0$,
then $r_i(t)$ would have to converge to a limit $r_i^*$ in $[\beta_i/a_i, \infty)$.
Going around the loop one would then obtain that all other variables converge to constant limit states.
This argument uses the fact that  if  a continuous function $h$ satisfies $h(t) \to 0 \, (t \to \infty)$ and $\lambda >0$, then also
$\int_0^t \exp[-\lambda(t-s)]h(s) \, ds \to 0 $ as $ t \to\infty$. Finally one would obtain an equilibrium state with
$r_i^* \geq \beta_i/a_i$, contradicting the above statement on the location of  equilibria.
Thus, the set of states with all values positive and $r_i \in (0, \beta_i/a_i) $  is also attracting. It follows  by an analogous
argument that the set of states with all values positive and $p_i \in (0, c_i/b_i) $ is also invariant and attracting.
Part a) is proved.

 Ad b): The equations for an equilibrium are
\begin{equation}
r_i^* = \frac{\beta_i}{a_i} f_i(p_{i-1}^*), \; p_i^* = \frac{c_i}{b_i}r_i^*, \; i = 1,  \ldots , n \; \text{ (mod } n). \label{61.1} \end{equation}
Hence, an equilibrium is given by a solution $(r_1^*,   \ldots , r_n^*)$ of
\[r_i^* = \frac{\beta_i}{a_i}
f_i[\frac{c_{i-1}}{b_{i-1}}r_{i-1}^*], \;  i = 1,  \ldots , n \text{ (mod } n),  \]
 and the corresponding  values of the $p_i^*$ as in \eqref{61.1}.
Setting $  F_i(x) :=  \frac{\beta_i}{a_i} f_i[\frac{c_{i-1}}{b_{i-1}}x]$,
the last equations are  equivalent to the single  fixed point equation
\begin{equation}  r_n^* = F_n(F_{n-1}(  \ldots (F_1(r_n^*)   \ldots ), \label{61.2} \end{equation}
with the values of the remaining $ r_i^* $ then determined accordingly.
Consider now the function $F^* := F_n \circ F_{n-1} \circ   \ldots  \circ F_1$, and recall  that at least one of the $f_i$  is decreasing and hence takes only positive values (see \eqref{deacinc}). It follows that
$F^*(0)  >0$.  Further,  since  $ f_1(x)  \to 1 \; (x \to \infty)$, the limit
$ \lim_{x \to \infty} F_1(x)$ exists, and hence also $ \lim_{x \to \infty}  F^*(x)$ exists, and $F^*(x) - x  \to -\infty \; (x\to \infty)$.  It follows (by the intermediate value theorem) that $F^*$ has a fixed point
$r^* \in (0, \infty)$. This fixed point is unique, because $F^*$ is strictly monotone
(actually, decreasing) as a composition of strictly monotone functions.
Thus, we obtain a unique solution of \eqref{61.2}, and hence a unique equilibrium, which
clearly satisfies the restrictions from a).
\end{proof}

We conclude this last  section with a theorem that results from combination  of Theorem  \ref{PerSolMfd}
(applied to system \eqref{genereg}  after transformation  to the form \eqref{standarduni}, with $ \tau = 1$) and  Lemma \ref{lem61}.

\begin{thm}\label{thm62} Consider system \eqref{genereg} with an odd number of decreasing $f_i$.

\begin{enumerate}
\item[a)] This system has a unique equilibrium $(r_1^*, p_1^*,   \ldots , r_n^*, p_n^*) \in (0, \infty)^{N}$,
where $N := 2n$, and associated to  system \eqref{genereg}  and this equilibrium are  the  values
$K, \mu_1,   \ldots , \mu_N$ as defined by \eqref{mutrans} and \eqref{Ktrans}.

\item[b)] Let  $K_c$ and $K_u$ be  the oscillation and the stability border for  the   system transformed to
the form \eqref{standarduni} (with $ \tau = 1$), described in Theorem \ref{thm52}.
If $ n > 1$ then both  possibilities $K_c < K_u$ and $ K_c > K_u$ can occur.

\item[c)] If $K > K_u $ then the transformed system has  a two-dimensional invariant manifold $W$ with  a bounding periodic orbit as  described in Theorem \ref{PerSolMfd}.
The corresponding manifold and periodic orbit for  the original system \eqref{genereg}
satisfy  the restrictions from part a) of Lemma \ref{lem61}.
\end{enumerate}
\end{thm}
\begin{proof} Part a) follows from  Lemma \ref{lem61} and the transformations described before  that lemma.
Part b) follows from  part c) of Theorem \ref{thm52}, since $n >1$ implies that $N = 2n \geq 3$.

Ad c): If $K > K_u $ then part a) of Theorem \ref{thm52} shows that  Theorem \ref{PerSolMfd} applies
and gives the manifold with bounding periodic orbit.  The  corresponding manifold
for  the non-transformed system is, of course, part of the global attractor, and thus the
 restrictions on the values  of the $p_i$ and $ r_i$  within the closure of this manifold are clear from Lemma  \ref{lem61}.
\end{proof}

\begin{rem}{(Final remarks.)}
  \begin{enumerate}
\item[1)] To our knowledge, so far the method of attractor location from \cite{IvaSharko}, which was applied to the unidirectional special case (system \eqref{standarduni}) in Theorem  \ref{attrloc},
has not been extended to the more  general coupling structure of system \eqref{standardfb}
(this is an open  question).

\item[2)]
The construction of the manifold $W$ in Theorem  \ref{PerSolMfd} for system \eqref{standardfb}, heavily using the results of
\cite{MPS1} and   \cite{MPS2},  was guided by the  construction presented  by H.-O.  Walther in  \cite{Wal91} for the case of
a scalar equation.
The latter was extended in  \cite{Wal95} to a description of the whole attractor within the lowest oscillation class
(a two-dimensional graph with possibly  several nested periodic solutions).
This result refers to the same class  of scalar equations as in \cite{Wal91}.
It remains  an open problem to obtain  an analogous description for system  \eqref{standardfb} and  the global attractor within
the level set associated to level 1 of the Lyapunov functional.
\end{enumerate}
\end{rem}

\section*{Acknowledgments}
 This work was partially supported by the Alexander von Humboldt Foundation (Germany).
It was initiated during A. Ivanov's visit to Justus-Liebig-Universit\"at Gie\ss en in June - August 2023.
A.I. is thankful to both institutions for their hospitality, accommodation, and support.

We further  thank two anonymous referees for their diligent work that helped to improve the presentation.

\medskip
Received xxxx 20xx; revised xxxx 20xx; early access xxxx 20xx.
\medskip

\end{document}